\documentclass[11pt]{article}
\usepackage{geometry}
\usepackage{amsfonts,amsmath,amsthm,amssymb}
\usepackage[dvipsnames]{xcolor}
\usepackage[hyperfootnotes=false]{hyperref}
\usepackage{array,float}
\usepackage{graphicx}
\usepackage{caption}
\usepackage{enumitem}
\usepackage{mathtools}
\usepackage{tikz-cd}
\usepackage[noadjust]{cite}
\usepackage{graphicx}
\usepackage[normalem]{ulem}
\usepackage{cleveref}
\usepackage{comment,verbatim}

\setlist{itemsep=0pt,topsep=5pt}

\theoremstyle{plain}
\newtheorem{theorem}{Theorem}[section]
\newtheorem{lemma}[theorem]{Lemma}
\newtheorem{corollary}[theorem]{Corollary}
\newtheorem{proposition}[theorem]{Proposition}
\newtheorem{conjecture}[theorem]{Conjecture}

\theoremstyle{remark}
\newtheorem{remark}[theorem]{Remark}

\theoremstyle{definition}
\newtheorem{definition}[theorem]{Definition}

\numberwithin{equation}{section}
\numberwithin{table}{section}
\numberwithin{figure}{section}

\newcommand{\R}{\mathbb{R}}
\newcommand{\C}{\mathbb{C}}
\newcommand{\Z}{\mathbb{Z}}
\newcommand{\Q}{\mathbb{Q}}
\DeclareMathOperator{\Gal}{\mathrm{Gal}}
\DeclareMathOperator{\gl}{\mathrm{GL}}
\DeclareMathOperator{\Log}{\mathrm{Log}}
\DeclareMathOperator{\rk}{\mathrm{rk}}
\DeclareMathOperator{\CM}{\mathrm{CM}}
\DeclareMathOperator{\non-CM}{\mathrm{non-CM}}
\DeclareMathOperator{\rel}{\mathrm{rel}}
\DeclareMathOperator{\N}{\mathrm{N}}
\DeclareMathOperator{\Tr}{\mathrm{Tr}}

\newcommand\nonumberfootnote[1]{
    \begingroup
    \renewcommand\thefootnote{}\footnote{#1}
    \addtocounter{footnote}{-1}
    \endgroup
}

\title{On the Geometry and Shapes of Rank 2 Log Unit Lattices}

\author{
  Cruz, Jose\\
  \texttt{jose.cruz@ucalgary.ca}
  \and
  Holmes, Erik\\
  \texttt{erikholm@iu.edu}
  \and 
  Jalalvand, Fatemeh \\
  \texttt{fatemeh.jalalvand@ucalgary.ca}
  \and
  Nunez Lon-Wo, Enrique\\
  \texttt{enrique.nunezlon.wo@mail.utoronto.ca}
  \and
  Scheidler, Renate\\
  \texttt{rscheidl@ucalgary.ca}
  \and
  Tran, Ha T. N.\\
  \texttt{htran2@ualberta.ca}
}

\begin{document}

\maketitle

\begin{abstract}
Every number field canonically gives rise to two lattices: its ring of integers and its log unit lattice. While the shapes of the former have undergone extensive research, far less is known about the shapes of the latter, referred to as \emph{unit shapes}. This paper presents an in-depth analysis of the unit shapes of number fields with unit rank~2. Our first main result characterizes, in many cases, the location of unit shapes within the fundamental domain of the space of rank~2 lattice shapes in terms of the Galois group of the field's Galois closure, and determines when these unit shapes are transcendental. Next, we establish that the unit shape uniquely determines the field up to isomorphism for totally imaginary $D_6$ non-CM sextic fields; this result fails in the CM case. Finally, for certain subfamilies of $D_6$ non-CM imaginary sextics, we offer a simple sufficient condition for their log unit lattices to be orthogonal and provide lower bounds on the proportion of fields with orthogonal log unit lattice.

 \nonumberfootnote{\emph{Keywords:} Rank 2 lattice, log unit lattice, fundamental units, unit shape, $D_6$ sextic field.}
\nonumberfootnote{\emph{2020 Mathematics Subject Classification:} 11H06, 11P21, 11R27, 11R21, 11Y40, 11R32.}
\end{abstract}

\tableofcontents


\section{Introduction}
Lattices occupy a central place in modern mathematics: from Minkowski's geometry of numbers, to the sphere packing problem in Euclidean space, to arithmetic statistics. More recently, in the wake of post-quantum cryptography, lattice-based cryptosystems have brought lattices to the forefront of public-key cryptography, drawing attention to these geometric structures well beyond pure mathematics. Much of this attention has been directed at lattices attached to \emph{additive} structures --- integral lattices of number fields, ideal lattices, root lattices --- while comparatively little is known about their multiplicative counterparts. There has, however, been a recent surge of work on the multiplicative lattices arising from number fields, the \emph{log unit lattices}, obtained by applying the logarithmic Minkowski embedding to the unit group modulo torsion. These lattices, and in particular their geometry, are the focus of our work.

Several threads motivate an in-depth investigation of log unit lattices. First, computing the class number, the unit group, or the log unit lattice of a number field is among the central problems of computational number theory, and these computations are hard: the fastest classical algorithms run in subexponential time. Partial knowledge of the geometry of the log unit lattice can yield speedups, as Schoof and Washington demonstrated for quintic fields \cite{schoof1988quintic}. Second, the size function of a number field is an arithmetic analogue of the dimension of the Riemann--Roch space of a divisor on an algebraic curve \cite{geer-schoof}, and understanding the geometry of the log unit lattice is crucial for proving the Schoof--van der Geer conjecture on the maximum of this size function \cite{TranTian2018SizeFunction, tran2023size}. Third, several cryptographic schemes rely on the hardness of the Short Generator Principal Ideal Problem, often in cyclotomic fields \cite{CGS14}, which becomes solvable in polynomial time once the underlying log unit lattice is understood \cite{CDPR16, HWB17}.

A recurring question in algebraic number theory asks whether a chosen invariant determines a number field up to isomorphism. On the additive side, the discriminant is a complete invariant for quadratic fields but loses this power in higher degree. The lattice arising from the ring of integers encodes more information than the discriminant, and its shape is another natural invariant to consider. In certain cases it is shown that the shape is a complete invariant in several natural families of degree exceeding~2. Yet it too can fail within a family; for example, every Galois cubic field has the same hexagonal shape \cite{terr1997}. On the multiplicative side, one is led to the analogous questions for the unit group, where the regulator (the covolume of the log unit lattice) plays the role of the discriminant, and the lattice itself remembers strictly more than its covolume. The CM fields make a natural test case: complex conjugation collapses their unit rank to that of the totally real subfield, so their log unit lattice is essentially inherited from this subfield. Here, one might expect the \emph{unit shape}, i.e.\ the shape of the log unit lattice, to be too rigid to tell fields apart. One of the problems investigated herein is to determine just how strong an invariant the unit shape really is, both for CM fields and beyond.

The arithmetic question of the strength of the unit shape as an invariant is complemented by a geometric question which accounts for the bulk of our results. Viewed as a point of the moduli space of all shapes of rank $n$ lattices, where does the unit shape lie? Is its location restricted to particular subspaces, and does it lie in the interior of the space or on its boundary? In rank two, one can ask whether the lattice is \emph{orthogonal} (i.e.\ has an orthogonal $\mathbb{Z}$-basis) or \emph{well-rounded} (i.e.\ has a $\mathbb{Z}$-basis consisting of shortest vectors) --- the distinguished forms that mark out special positions in the moduli space. As a complex number, is the point algebraic or transcendental, and does a whole family of fields fill out a region of the space or trace only a thin, explicitly describable locus?

It is natural to pursue these questions within families of fixed signature, often also fixing the Galois group~$G$ of the Galois closure. Integral lattices of of degree $n$ fields give rise to lattices of rank $n-1$ as one projects away from the common subring~$\Z$. For generic fields (those with $G \cong S_n$), the corresponding shapes are conjectured to equidistribute in the space $\mathcal{S}_{n-1}$ of shapes of rank $n-1$ lattices. In contrast, non-generic families are typically confined to explicit subvarieties of $\mathcal{S}_{n-1}$. 
Analogous distribution phenomena appear for log unit lattices \cite{HHV, RNT}. The smallest non-trivial setting is the case of rank~$2$, where $\mathcal{S}_2$ is the familiar modular domain and each shape is a single point of the upper half-plane. This is the arena of the present paper. We give a detailed account of rank~$2$ unit shapes --- where they lie within $\mathcal{S}_2$ and, in the central case of the totally imaginary sextic fields with Galois group $D_6$ (the dihedral group of order~$12$), how completely the shape determines the field.

\subsection{Our main results}

Our work utilizes tools from algebraic number theory, representation theory, lattice theory, and algebraic geometry. The results fall into three parts: a classification of rank~$2$ unit shapes by Galois group; an investigation into the geometry of unit lattices the totally imaginary $D_6$ family, where both the power of the unit shape as an invariant and the geometry of the resulting shape set are governed by a single dichotomy; and finally explicit orthogonality and density results for a family of pure sextic fields.

\subsubsection{A classification of rank \texorpdfstring{$2$}{2} unit shapes.} \label{ss:unit-rank-classification}

Adopting the notation from \cite{RNT}, let $\Omega(G, r, s)$ denote the family of unit shapes of fields with signature $(r,s)$ whose Galois closure has Galois group $G$, and $\Omega_{\CM}(G, r, s)$ (resp.\ $\Omega_{\non-CM}(G, r, s)$) the subfamily of shapes of CM fields (resp.\ non-CM fields) in $\Omega(G, r, s)$.
Our first main result, \Cref{theorem: classification or rank 2}, locates the unit shapes of unit rank~$2$ number fields within~$\mathcal{S}_2$ and determines their transcendentality according to the Galois group $G$. For example, the set $\Omega(D_6,0,3)$ is contained in $\mathcal{S}_2$. \Cref{fig: D6_sextic_plot} shows a plot of some members of $\Omega(D_6,0,3)$ inside the standard fundamental domain $\mathcal{D}$ (see \Cref{ss:lattices}).

\begin{figure}[ht]
    \centering{
    \includegraphics[scale=.35]{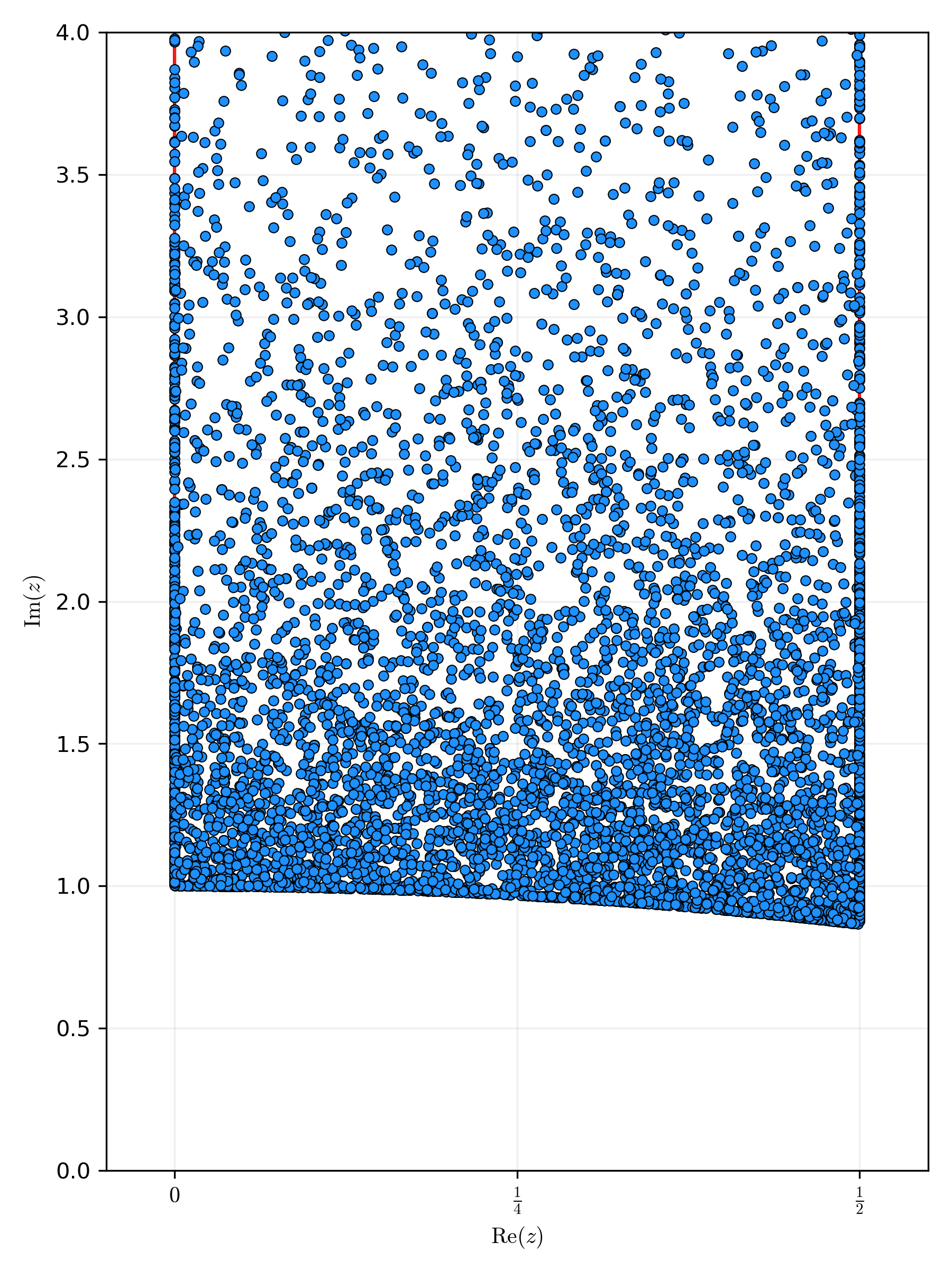}}
    \caption{Unit shapes of $D_6$ sextics of signature $(0,3)$ plotted in the fundamental domain $\mathcal{D}$ of $\mathcal{S}_2$.}
    \label{fig: D6_sextic_plot}
\end{figure}

The admissible groups for each signature of a unit rank 2 number field are enumerated in \Cref{prop:Galois groups rank 2}. It organizes these families into three sharply distinct regimes:
\begin{itemize} 
    \item Fields whose Galois group carries an order $3$ automorphism acting non-trivially on the unit lattice --- among them the cases $G \in \{C_3, C_6, S_3, {S_3\times C_3}\}$ --- have hexagonal unit shape.
    \item The families $\Omega(D_4,2,1)$ and $\Omega_{\non-CM}(G,0,3)$ 
      for $G \in \mathcal{G} \coloneq \{D_6,\, S_4, \ A_4 \times C_2,\, S_4 \times C_2\}$ consist of transcendental shapes confined to the boundary $\partial \mathcal{S}_2$. 

   \item Assuming the Weak Schanuel Conjecture (\Cref{conj: algebraic independence of logs}), the families $\Omega(S_3,3,0)$  and $\Omega_{\CM}(G,0,3)$ 
    for $G \in \mathcal{G}$ lie in the interior of $\mathcal{S}_2$ and are $\overline{\Q}$-generic: no proper $\overline{\Q}$-algebraic subvariety of $\mathcal{S}_2$ contains them, so they are as algebraically unconstrained as possible.
\end{itemize}

The third point stands in sharp contrast to the $D_5$ family of \cite{HHV}, whose unit shapes are constrained to an explicit algebraic hypercycle in $\mathcal{S}_2$; our result shows that this kind of algebraic rigidity is \emph{not} a general feature of rank~$2$ unit shapes (see \Cref{remark: comparison of results}). The classification leaves open the cases of fields with no proper subfields and of $S_3^2$- and $S_3^2 \rtimes C_2$-sextic fields; computational evidence, including a PSLQ search that detects no algebraic relations, supports the conjecture that these too are transcendental and interior (\Cref{conj:remaining Galois groups}).

\subsubsection{The totally imaginary \texorpdfstring{$D_6$}{D6} family.}
For a totally imaginary sextic field $F$ with $G = D_6$, the unique cubic subfield $F_3$ is either totally real --- in which case $F$ is CM and $F_3$ has unit rank~$2$ --- or complex, in which case $F$ is non-CM and $F_3$ has unit rank~$1$. By \Cref{theorem: classification or rank 2}, this single distinction already separates $\Omega(D_6,0,3)$ into the interior CM part $\Omega_{\CM}(D_6,0,3)$ and the boundary non-CM part $\Omega_{\non-CM}(D_6,0,3)$ in $\mathcal{S}_2$. It also governs the following two questions we pursue for this family.
\begin{itemize} \itemsep 2pt
    \item \emph{How strong is the unit shape as an invariant?} The two subfamilies behave in opposite ways. In general, in the CM case, the unit shape is a weak invariant: if $K$ is a fixed totally real number field, then for all but finitely many CM extensions $F/K$, the fields $K$ and $F$ have the same unit shape (\Cref{prop: the shape is weak}, with an explicit family in \Cref{prop:CMfields}). Hence, infinitely many non-isomorphic CM $D_6$-sextic fields share a single shape (\Cref{prop: real cubic shapes are CM shapes}). In the non-CM case, the unit shape is, by contrast, conjecturally a \emph{complete} invariant: assuming the Four Exponentials Conjecture (\Cref{conj: four exponentials problem}), two totally imaginary $D_6$-sextic fields with complex cubic subfield are isomorphic if and only if their unit shapes coincide (\Cref{theorem: Completeness of the shape}).
 
    \item \emph{What does the shape set look like inside $\mathcal{S}_2$?} We show that $\Omega(D_6,0,3)$ is non-discrete. On the CM side, the shapes of totally real $S_3$-cubic fields embed into $\Omega_{\CM}(D_6,0,3)$ (\Cref{prop: real cubic shapes are CM shapes}) and accumulate at the hexagonal point. On the non-CM side, we exhibit the infinite family $F = \Q\!\left(\sqrt[6]{-27(D^6 \pm 1)}\right)$ whose unit shapes lie on the right boundary of $\mathcal{S}_2$ and converge to the hexagonal shape as $D \to \infty$ (\Cref{theorem: limitpoint}); strikingly, the sequence converges to a well-rounded shape although no member of it is itself well-rounded.
\end{itemize}
We emphasize that non-discreteness --- the presence of infinitely many distinct shapes in each subfamily --- is a first step only. The finer questions of the density and equidistribution of unit shapes within $\mathcal{S}_2$ remain open, and motivate much of this work; we return to them in \Cref{sec:conclusion}.

\subsubsection{Orthogonality and density for pure sextics.}
For the pure sextic fields $F$ built from the compositum of a pure cubic $\Q(\sqrt[3]{D^3 \pm 1})$ and an imaginary quadratic field $\Q(\sqrt{-d})$, we determine an explicit quartic polynomial, depending only on $D$, whose lack of an integer root modulo a suitable large factor of $d$ forces the unit shape of~$F$ to be orthogonal, or equivalently, to lie on the left boundary of $\mathcal{S}_2$ (\Cref{prop:orth}). From this construction, we produce infinitely many infinite subfamilies of orthogonal log unit lattices, both for fixed $d$ with $D$ varying (\Cref{cor:fixd_varym}) and for $F = \Q(\sqrt[3]{2}, \sqrt{-d})$ with $d$ varying (\Cref{cor:m=2varyd}). For the family $\Q(\sqrt[3]{2}, \sqrt{-d})$ we go further and quantify how common orthogonality is: at least $68\%$ of these fields have orthogonal log unit lattice (\Cref{thm:2dens}), and when $d = p$ is prime the density is at least $3/8$ (\Cref{cor:density_prime}), via a Chebotarev argument applied to the Galois group of the quartic. 

\subsection{Previous work on log unit lattices} \label{subsection: previous work}

We provide an overview of existing literature on the geometry and shapes of log unit lattices, organized by unit rank.

In rank~$2$, the log unit lattices of cyclic cubic fields are hexagonal. This follows from the more general fact that any rank~$2$ lattice admitting an order $3$ automorphism is hexagonal, established in \cite{terr1997} by Terr and in \cite{TranTian2018SizeFunction} by Tran and Tian. Hexagonality was proved for imaginary cyclic sextic fields by Tran \textit{et al} in \cite{tran2023size}, and the cyclic cubic case was also discussed by Azpeitia-Tellez \textit{et al} in \cite{TPS2021}. David and Shapira \cite{DS_unitlattices} investigated the unit shapes of totally real cubic fields and conjectured that they are dense in $\mathcal{S}_2$, providing evidence via a generalization of the simplest cubic fields. Chari \emph{et al.}~\cite{RNT} explored unit shapes of $D_4$-quartic fields of signature $(2,1)$, showing that they are transcendental points on the boundary of $\mathcal{S}_2$ and constructing explicit algebraic limit points arising from families of such fields.

In rank~$3$, Azpeitia-Tellez \textit{et al} \cite{TPS2021} examined the geometry of log unit lattices of real biquadratic fields and determined when these lattices are orthogonal.

For higher rank, Cruz \cite{cruz2020well} investigated the shortest vectors of log unit lattices of totally real cyclic fields, providing a parametrizing space for the Gram matrices of those admitting a Minkowski unit (see \Cref{def:minkowskiunit}), characterizing their well-roundedness for extension degree at most~$7$, and computing limit points of several infinite families. In \cite{cruz2025classificationtotallyrealnumber}, Cruz proved that two real Galois number fields are isomorphic if and only if their log unit lattices are isometric, and that for extension degree at least~$4$, they are isomorphic if and only if their unit shapes agree. Harron \textit{et al.} \cite{HHV} analyzed unit shapes of $D_p$ number fields of signature $[1,(p-1)/2]$. They showed that these unit shapes lie on an explicit hypersurface in the space of shapes, described as a finite union of torus orbits; in the rank~$2$ case the shapes lie on a single fixed hypercycle. We note that, although this observation was not explicitly stated in that source, the same arguments can also be applied to imaginary Galois $D_p$ fields. 

Lastly, for arbitrary number fields, Kessler \cite{kessler1991minimum} provided a bound on the minimum (i.e.\ the length of any non-zero shortest vector) of the log unit lattice in terms of the unit rank and the degree of the field.


\section{Preliminaries} \label{sec:background}

\subsection{Notation}\label{Notation}

In the context of lattices, we require the following notions and notation. 

\begin{center}
    \begin{tabular}{|c|l|} \hline
    \textit{Symbol} & \textit{Meaning} \\ \hline
    $\rk(L)$ & Rank of a lattice $L$ \\
    $[L]$ & Shape of a lattice $L$ \\
    $\mathcal{G}_m$ & Space of rank $m$  Gram matrices \\
    $\mathcal{S}_m$ & $\mathcal{S}_m = \text{GO}_m(\R)\backslash \gl_m(\R) / \gl_m(\Z)\cong 
    \R^\times\setminus \mathcal{G}_m/ \gl_m(\Z)$, \\
    & space of shapes of rank $n$ lattices 
   
             \\ \hline
    \end{tabular}
\end{center}

In the setting of a fixed number field $F$, we will use the notation in the table below. Most of the quantities associated to $F$ will be identified by the subscript $F$; if the context is clear, we may omit this subscript. 

\begin{center}
    \begin{tabular}{|c|l|} \hline
         \textit{Symbol} & \textit{Meaning} \\ \hline
         $F$ & Number field \\
         $F^{\times}$ & $F\backslash\{0\}$ \\
         $(r,s)$ & Signature of $F$; in particular, $F$ has degree $n = r + 2s$ \\
         $\tilde{F}$ & Galois closure of $F$ \\
         $G = \Gal(\tilde{F}/\Q)$ & Galois group of $\tilde{F}/\Q$ \\
          $\mathcal{F}(G, r, s)$ & Set of fields $F$ with signature $(r,s)$ and $\Gal(\tilde{F}/\Q) = G$ \\
          $\mathcal{F}_{\CM}(G, r, s)$ & Subset of fields in $\mathcal{F}(G, r, s)$ with complex multiplication
         \\
         $\mathcal{F}_{\non-CM}(G, r, s)$ & Subset of fields in $\mathcal{F}(G, r, s)$ without complex multiplication \\ 

        $\Omega(G,r,s)$ & Set of unit shapes of fields in $\mathcal{F}(G, r,s)$ \\

        $\Omega_{\CM}(G,r,s)$ & Set of unit shapes of fields in $\mathcal{F}_{\CM}(G, r, s)$\\

        $\Omega_{\non-CM}(G,r,s)$ & Set of unit shapes of fields in $\mathcal{F}_{\non-CM}(G, r, s)$\\
        
         $\sigma_1, \ldots , \sigma_r$ & Real embeddings $F \hookrightarrow \mathbb{R}$ \\
         $\tau_1, \overline{\tau}_1, \ldots, \tau_s, \overline{\tau}_s$ & Complex embeddings $F \hookrightarrow \mathbb{C}$, where $\overline{\tau}_i$ is the complex conjugate of $\tau_i$ \\
         $\Delta_F$ & Discriminant of $F$ \\
         $O_F$ & Ring of integers of $F$ \\
         $O_F^\times$ & Unit group of $O_F$ \\
         $\mu_F$ & Roots of unity in $F$ and torsion part of $O_F^\times$ \\
         $E_F$ & $E_F \cong O_F^\times / \mu_F$, torsion-free part of $O_F^\times$, of rank $r+s-1$ \\
         $\Log_F$ & Logarithmic Minkowski embedding of $F^\times$ into $\R^{r+s}$ \\
         $\Lambda_F$ & $\Lambda_F = \Log_F(O_F^\times) = \Log_F(E_F)$, log unit lattice of $F$ \\ \hline
    \end{tabular}
\end{center}

In connection with relative number field extensions $F/K$, we employ the following notation. Again, if the context is clear, we may omit subscripts. 

\begin{center}
    \begin{tabular}{|c|l|} \hline
         \textit{Symbol} & \textit{Meaning} \\ \hline
         $F/K$ & Relative extension of number fields \\ 
         
         $N_{F/K}$ & Norm map $F \longrightarrow K$ \\
         $\iota$ & Embedding $\iota : \Lambda_K \rightarrow \Lambda_F$. \\ \hline
    \end{tabular}
\end{center}

If in addition $F/K$ is a Galois extension, we use the following symbols.

\begin{center}
    \begin{tabular}{|c|l|} \hline
         \textit{Symbol} & \textit{Meaning} \\ \hline
         $H = \Gal(F/K)$ & Galois group of $F/K$  \\
         $\gamma$ & Element in $\Gal(F/K)$ \\
         $\Lambda_{F/K}^t$ & $\Lambda_{F/K}^t = \{ \Log_F(u) : \mbox{$u \in O_F^\times$ and $u^{1-\gamma} \in \mu_F$ for all $\gamma \in \Gal(F/K)$}\}$ \\
         $\Lambda_{F/K}^r$ & $\Lambda_{F/K}^r = \{ \Log_F(u) : \text{$u \in O_F^\times$ and $\N_{F/K}(u) \in \mu_K$}\}$ \\
         $v_t, v_r$ & Generators of $\Lambda_{F/K}^t, \Lambda_{F/K}^r$ when $\Lambda_{F/K}^t$ and $\Lambda_{F/K}^r$ have rank 1 \\
         $u_t, u_r$ & Respective pre-images of $v_t, v_r$ in $O_F^\times$ under $\Log_F$
        \\ \hline
    \end{tabular}
\end{center}

\subsection{Lattices and their shapes} \label{ss:lattices}
Let $(V, \langle \cdot, \cdot \rangle)$ be a finite-dimensional real inner product space and let $\|.\| = \sqrt{\langle \cdot, \cdot \rangle}$ be its corresponding norm. A \emph{lattice} in $V$ is a discrete subgroup of $V$. Every lattice $\Lambda$ in in $V$ is of the form 
\[ \Lambda=\mathbb{Z}v_1 + \mathbb{Z}v_2 + \cdots + \mathbb{Z}v_m = \left\{\sum_{i=1}^{m}a_iv_i : a_i\in \mathbb{Z}\right\},  \]
for some positive integer $m$, where $\mathcal{B}=\{v_1,v_2,...,v_m\}$ is a set of linearly independent vectors in $V$. The quantity $m$ is the \emph{rank} of $\Lambda$, and $\mathcal{B}$ is a \textit{basis} of $\Lambda$. In case $m=\dim(V)$, the lattice $\Lambda$ is said to be of \textit{full rank}.

The \emph{Gram matrix} of a basis $\mathcal{B} = \{v_1,v_2,...,v_m\}$ of a rank $m$ lattice $\Lambda$ is the positive definite symmetric matrix $(\langle v_i, v_j \rangle)_{1 \le i, j \le m}$ of pairwise inner products of the vectors in $\mathcal{B}$. By abuse of terminology, we will at times speak of a Gram matrix of $\Lambda$ when we mean the Gram matrix of some fixed basis of $\Lambda$.

 The $i^{th}$-\textit{successive minimum} of $\Lambda$ is the quantity \[\lambda_i(\Lambda)=\inf \big \{r > 0 : \dim\left(\mathrm{span}_\R\{v\in \Lambda : \|v\|\leq r\}\right) \geq i \big \}.\]
It is the smallest positive $r$ such that the ball of radius $r$ contains $i$ linearly independent vectors of $\Lambda$. If a basis $B =\{v_1,\dots, v_m\}\subseteq \Lambda$ satisfies $\|v_i\|=\lambda_i$, $1\leq i\leq m$, then we say that $B$ \emph{achieves} the successive minima of $\Lambda$.
 
 The \textit{smallest norm} of $\Lambda$ is the quantity $\|\Lambda\| = \min_{0\ne u \in \Lambda}\|u\|^2 \in \mathbb{R}^{>0}$, which exists since $\Lambda$ is discrete. It satisfies $\|\Lambda\|=\lambda_1^2(\Lambda)$. Any vector $u \in \Lambda$ with $\|u\|^2=\|\Lambda\|$ is a \emph{shortest vector} of $\Lambda$. 

Two lattices $\Lambda_1, \Lambda_2$ are said to be \textit{isometric} if there exists an orthogonal transformation $T\in \mathrm{O}(V)$ such that $ T \Lambda_1=\Lambda_2$. 
They are called \textit{homothetic} (or \textit{similar}) if there exists a non-zero constant $c\in \mathbb{R}$ and an orthogonal transformation $T\in \mathrm{O}(V)$ such that $ cT \Lambda_1=\Lambda_2$.

\begin{definition} \label{def:WR} \mbox{ }
\begin{enumerate} \itemsep -10pt
\item[(i)] A lattice $\Lambda$ in $V$ is \emph{well-rounded} if its shortest vectors span $\Lambda \otimes_\Z \mathbb{R}$  and \emph{orthogonal} if it has a basis consisting of pairwise orthogonal vectors.     
\item[(ii)] A rank $2$-lattice $L$ is called \emph{hexagonal} if it is similar to the lattice $\mathbb{Z}\begin{bmatrix}
            1\\
            0
        \end{bmatrix} + 
        \mathbb{Z}\begin{bmatrix}
            1/2\\
            \sqrt{3}/2
        \end{bmatrix}$.
\end{enumerate}
		
\end{definition} 

We now restrict to the full rank case $m = \dim(V)$ and identify $V$ with $\R^m$. We write any $v \in V$ as a column vector and denote by $v^T$ its transpose; analogous notation is used for $m \times m$ matrices over $\R$. Denote by~$\mathcal{G}_m$ the set of symmetric positive definite matrices of rank $m$, or equivalently, the space of Gram matrices of rank $m$ lattices.

\begin{definition} \label{def:latticeshape}
The \emph{shape} of $\Lambda$, denoted $[\Lambda]$, is the equivalence class of $\Lambda$ up to scaling, rotation and reflection. The \emph{ambient moduli space $\mathcal{S}_{m}$ of shapes of rank $m$ lattices} is  the set of double cosets %
    \[  \mathcal{S}_{m} = \text{GO}_m(\R) \backslash \text{GL}_m(\R)/\text{GL}_m(\Z), \] 
    where $\text{GO}_m(\R)$ is the group of scaled orthogonal transformations on $\mathbb{R}^m$. The space $\mathcal{S}_m$ may also be identified with the set of rank $m$ Gram matrices up to change of basis and scaling, i.e.
    \[  \mathcal{S}_m = \R^\times \backslash \mathcal{G}_m/ \text{GL}_m(\Z).    \]
\end{definition}

In the special case $m = 2$, the space of rank 2 lattice shapes, up to similarity, is naturally identifiable with the quotient space GL$_2(\Z)/\mathcal{H}$, where $\mathcal{H}$ is the complex upper half plane and the action of $\mathrm{GL}_2(\Z)$ on $\mathcal{H}$ is given by
\[
\begin{bmatrix}
    a&b\\
    c&d
\end{bmatrix}\cdot (x+iy)=\frac{ax+iby}{cx+idy}, \qquad\qquad \begin{bmatrix}
    a&b\\
    c&d
\end{bmatrix}\in \mathrm{GL}_2(\Z), \ x+iy\in \mathcal{H}.
\]
 A fundamental domain for this action is given by the region $$\mathcal{D}\coloneq\{ z = x + iy \in \mathcal{H} : |z| \ge 1, 0\le x \le 1/2\}\subseteq \mathcal H.$$
 Every lattice $\Lambda$ of rank 2 has a basis $\mathcal{B}\coloneq\{v_1, v_2\}$ achieving the successive minima of $\Lambda$ such that $\langle v_1, v_2 \rangle \ge 0$. Applying suitable rotations, reflections, and scalings to $\mathcal{B}$ yields a basis of a lattice similar to $\Lambda$ of the form $\left \{ \begin{bmatrix} 1 \\ 0 \end{bmatrix}, \begin{bmatrix} x \\ y \end{bmatrix} \right \}$ where $x+ iy \in \mathcal{D}$. On the other hand, the Gram matrix of $\mathcal{B}$ has the form
\[
{\|v_1\|^2}\begin{bmatrix}
    1&x\\
    x&x^2+y^2
\end{bmatrix}.
\]
So the similarity classes of rank 2 lattices are in bijection with the points $x + iy \in \mathcal{D}$ via the above procedure. Hence $x + iy \in \mathcal{D}$ can be viewed as the shape of $\Lambda$.

By virtue of this correspondence and with a slight abuse of terminology, we use $\mathcal{S}_2$ to interchangeably refer to the double quotient space of rank 2 lattice shapes GL$_2(\Z)/\mathcal{H}$ and the region $\mathcal{D}$. As usual, we refer to the region $\mathcal{D}$ inside $\mathcal{H}$ as the \emph{(standard) fundamental domain}.

\subsection{Log unit lattices}
Let $F/\Q$ be a number field. Let $\sigma_1, \ldots \sigma_r$ be the real embeddings of $F$ and $\tau_1, \overline{\tau}_1, \ldots , \tau_s, \overline{\tau}_s$ the pairs of complex embeddings of $F$, so $F$ has signature $(r,s)$ and degree $n = r+2s$. Dirichlet's unit theorem establishes the group of units of the ring of integers $O_F$ as a finitely generated abelian group of rank $r+s-1$, i.e.\ 
\[ O_F^\times = E_F \oplus \mu_F \quad \text{with} \quad E_F \cong \Z^{r+s-1}, \]
where $\mu_F$ is the group of roots of unity in $F$. The image of $O_F^\times$ (or of $E_F$) under the logarithmic Minkowski embedding
\begin{align*}
\text{Log}_F: F^\times & \longrightarrow \R^{r+s}   \\
\alpha & \longmapsto
\begin{bmatrix} \log|\sigma_1(\alpha)|, \ \cdots, \ \log|\sigma_r(\alpha)|, \ 2\log|\tau_1(\alpha)|, \ \cdots, \ 2\log|\tau_s(\alpha)|
\end{bmatrix}^T
\end{align*}
lies in the trace-zero hyperplane
\[ \mathbb{H} = \left\{ \begin{bmatrix} x_1, x_2, \cdots ,x_{r+s} \end{bmatrix}^T \in \mathbb{R}^{r+s} \text{\quad \Bigg \vert \quad} \sum_{i=1}^{r+s} x_i = 0 \right\}.\] 

In fact, it is a lattice of rank $r+s-1$ in $\mathbb{H}$. 

\begin{definition} \label{def:logunitlattice} \mbox{ }
\begin{enumerate}
    \item[(i)]   The \emph{log unit lattice} of a field $F$ is the lattice $\Lambda_F = \text{Log}_F(E_F)$ in $\mathbb{H}$.  The \emph{unit shape} of $F$ is the shape $[\Lambda_F]$ of $\Lambda_F$ in $\mathcal{S}_{r+s-1}$.
    \item[(ii)]  The unit shape of $F$ is said to be \emph{hexagonal} if $[\Lambda_F]$ is a hexagonal lattice. 
\end{enumerate}
   
\end{definition}
Note that the map $\Log_F$, and hence the lattice $\Lambda_F$, depends on the choice of ordering of the embeddings of $F$. However, different reorderings of the embeddings will result in isometric, hence similar, log unit lattices.

Finally, in our investigation of log unit lattices, we will require two special types of units. 

\begin{definition} \label{def:minkowskiunit}
Let $F$ be a field, $\epsilon \in O_F^\times$, and $C_\epsilon \subset \C$ the set of Galois conjugates of $\epsilon$. Then $\epsilon$ is a \emph{(strong) Minkowski unit} of~$F$ if~$C_\epsilon \cap F$ generates~$E_F$ and a \emph{weak Minkowski unit} of~$F$ if~$C_\epsilon \cap F$ generates a subgroup of~$E_F$ of finite index.
\end{definition}

In the literature, weak and strong Minkowski units are generally only defined for Galois number fields, in which case $C_\epsilon \cap F = C_\epsilon$. Here, we extend these notions to non-Galois fields. We recall the following classical result due to Minkowski.

\begin{lemma}[\cite{ElementaryandAnalyticofANT}, Theorem 3.26]\label{thm: existence of weak mink unit}
Every Galois number field has a weak Minkowski unit.  
\end{lemma}

\begin{definition} \label{def:relativeunit}
     Let $F/K$ be a proper field extension of number fields. A unit $\epsilon \in O_F^\times$ is a \emph{relative unit} of $F/K$ if its relative norm is a root of unity in $K$, i.e.\ $\N_{F/K}(\epsilon) \in \mu_K$.
\end{definition}

\subsection{A representation associated with the log unit lattice}\label{subsection: representation theory section}
We now introduce some useful tools through representation theory. The main results of this subsection are \Cref{lem: orthogonality of relative units} and \Cref{lem:vtvrortho}  which yield sufficient conditions for the vectors in a log unit lattice to be orthogonal.

Let $F/K$  be a Galois extension of number fields with Galois group $H \coloneq \Gal(F/K)$. The torsion-free part $E_F \cong O_F^\times/\mu_F$ of the unit group $O_F^\times$ is a $\Z[H]$-module, and it is common in the literature to write the $\Z[H]$-module action on $E_F$ as follows: given an element $\alpha=\sum_{g\in H}n_gg \in \Z[H]$ and a unit $u \in E_F$, write
\[u^{\alpha} =\prod_{g\in H}g(u^{n_g}). \]

The log unit lattice $\Lambda_F$ of $F$ inherits a $\Z[H]$-action from the one on $E_F$ via the logarithmic Minkowski embedding $\Log_F$, which we will write additively, so
\[
\alpha\cdot\Log_F(u)=\Log_F(u^\alpha).
\] 
Since $H$ acts on $\Lambda_F$ by permuting the embeddings of $F$ into $\C$, $H$ acts through orthogonal transformations on $\Lambda_F$, so the inner product space $V_F\coloneq\Q\otimes_\Z\Lambda_F$ defines an orthogonal representation of~$H$, i.e.\ a map $H\rightarrow \mathrm{O}(V)$. The trivial component $V_{F/K}^t$ in $V_F$ is the subrepresentation of $V_F$ given by
\[ V_{F/K}^t = \{v\in V_F : gv=v \mbox{ for all } g\in H\}.\] 
Since $V_F$ is an orthogonal representation of $H$, the orthogonal complement of $V_{F/K}^t$ in $V_F$ is also a subrepresentation of $V_F$ which we denote by $V_{F/K}^r$. Thsi yields an orthogonal decomposition of $V_F$ as
\[
V_F = V_{F/K}^t\oplus V_{F/K}^r.
\]
Define 
\begin{equation} \label{eq:orthorep}
\Lambda_{F/K}^t \coloneq \Lambda_F\cap V_{F/K}^t, \quad \Lambda_{F/K}^{r} \coloneq \Lambda_F\cap V_{F/K}^r.
\end{equation}
These lattices can be explicitly described as follows.

\begin{lemma}\label{lem:vtvrortho}
Let $F/K$ be a Galois extension with Galois group $H=\Gal(F/K)$, and let $\Lambda_{F/K}^r$ and $\Lambda_{F/K}^t$ be as given in \eqref{eq:orthorep}. Then
    \begin{align*}
    \Lambda_{F/K}^t &= \{ \Log_F(u) : \mbox{$u \in O_F^\times$ and $u^{1-\gamma} \in \mu_F$ for all $\gamma \in H$}\}, \\
    \Lambda_{F/K}^r &= \{ \Log_F(u) : \text{$u$ is a relative unit of $F/K$}\}.
    \end{align*}
   
Moreover, $\Lambda_{F/K}^r$ is orthogonal to $\Lambda_{F/K}^t$.    
\end{lemma}
\begin{proof}
    For brevity, write $\Lambda^r = \Lambda_{F/K}^r$ and $\Lambda^t = \Lambda_{F/K}^t$. The description of $\Lambda^t$ follows immediately from its definition, so it suffices to prove the result for $\Lambda^r$.
    By abuse of notation, we view the relative norm $\N_{F/K}$ from $F$ down to $K$ as an element of $\Z[H]$, and write 
    \[
    \N_{F/K}\coloneq\sum_{\sigma\in H}\sigma\in \Z[H].
    \]
     Since $H$ acts orthogonally on $\Lambda_F$, it follows that for every $w\in \Lambda_F$, we have
     \[
    \begin{aligned}
         w \in \Lambda^r &\ \Longleftrightarrow \ \langle v, w\rangle=0 \text{ \ for all } v\in\Lambda^t \text{ by \eqref{eq:orthorep}}\\ &\ \Longleftrightarrow \ [F:K]\langle v,w\rangle=0\text{ \ for all } v\in\Lambda^t\\
         &\ \Longleftrightarrow \ \langle \N_{F/K} \,v,w\rangle=0\text{ \ for all } v\in\Lambda^t\\
         &\ \Longleftrightarrow \ \langle v,\N_{F/K} \,w\rangle =0 \text{ \ for all } v\in\Lambda^t\\
         &\ \Longleftrightarrow \ \N_{F/K} \, w=0 \text{ \ (because $\N_{F/K} \,w\in \Lambda^t$).}
    \end{aligned}
    \]
    This proves the result.
\end{proof}

It will at times be useful to consider the log unit lattice of some field as a sublattice of the log unit lattice of an extension field. This is formalized in the following remark.

\begin{remark}\label{remark: sublattice vs lattice}
    For an extension of number fields $F/K$, there exists a unique injective homomorphism  of abelian groups $\iota:\Lambda_K\rightarrow \Lambda_F$ such that the  following diagram commutes:
    \[\begin{tikzcd}
	{O_K^\times} & {O_F^\times} \\
	{\Lambda_K} & {\Lambda_F}
	\arrow[hook, from=1-1, to=1-2]
	\arrow["{\Log_K}"', from=1-1, to=2-1]
	\arrow["{\Log_F}", from=1-2, to=2-2]
	\arrow["\iota"', dashed, from=2-1, to=2-2]
\end{tikzcd}\]
For ease of notation, by identifying $\Lambda_K$ with $\iota(\Lambda_K) = \Log_F(O_K^\times) \subseteq \Lambda_F$, we may view $\Lambda_K$ as a subgroup  of $\Lambda_F$ and write $\Lambda_K\subseteq \Lambda_F$. However, as a word of caution, care is needed when considering $\Lambda_K$ as a lattice rather than just an abelian group, since $\Lambda_K$ and $\iota(\Lambda_K)$ may have different shapes, even though they are isomorphic as groups. For example, if $F= \Q(\sqrt[8]{3})$ and $K=\Q(\sqrt[4]{3})$, then  $[\Lambda_K]\neq[\iota(\Lambda_K)]$, as the two shapes correspond to different points in $\mathcal{S}_2$. This will be explored in more detail in Section~\ref{sec:shapesD6}. 
\end{remark} 

\Cref{lem:vtvrortho} provides a sufficient condition for the log vectors of the units of two number fields to be orthogonal inside the log unit lattice of their composite. This result will be used extensively throughout the remainder of the paper.

\begin{corollary}\label{lem: orthogonality of relative units} Let $F$ and $F'$ be two number fields such that $F/(F\cap F')$ and $F'/(F\cap F')$ are Galois. Then $\Lambda_{F'}$ is a subgroup of $\Lambda_{FF'/F'}^t$ and $\Lambda_{F/F \cap F'}^r$ is a subgroup of $\Lambda_{FF'/F'}^r$. In other words, if $u$ is a relative unit of $F/F \cap F'$, then $\Log_F(u)$ is orthogonal to  $\Lambda_{F'}$ in $\Lambda_{FF'}$.
\end{corollary}

\begin{proof} The extension $FF'/F'$ is Galois. Since its Galois group $\Gal(FF'/F')$ fixes every element of $F'$, it is clear that that $\Lambda_{F'} \subseteq \Lambda_{FF'/F'}^t$.

The restriction map $\Gal(F F'/F')\longrightarrow \Gal(F/F\cap F')$ is an isomorphism, so $\N_{FF'/F'}(\theta) = \N_{F/F \cap F'}(\theta)$ for all $\theta \in F$. In particular, if $u \in O_F^\times$ is a relative unit of $F/F \cap F'$, then 
\[ \N_{FF'/F'}(u) = \N_{F/F \cap F'}(u) \in \mu_{F \cap F'} \subseteq \mu_{F'}, \]
so $u$ is also relative unit of $FF'/F'$. Thus, $\Lambda_{F/F \cap F'}^r \subseteq \Lambda_{FF'/F'}^r$.
\end{proof}

We finish this subsection with the following result stating that $\Lambda_K$ has a finite index in $\Lambda_{F/K}^t$ when we embed $\Lambda_K$ as a subgroup of $\Lambda_F$, and establising that $\Gal(F/K)$ acts trivially on $\Lambda_F$ if and only if $F/K$ is a CM-extension.
\begin{proposition}\label{prop: subfield in trivial part}
    Let $F/K$ be a Galois extension with Galois group $H\coloneq\Gal(F/K)$, and view $\Lambda_K$ as a subgroup of $\Lambda_F$ as in \Cref{remark: sublattice vs lattice}. Then $\Lambda_K\subseteq \Lambda_{F/K}^t$ and the index $[\Lambda_{F/K}^t:\Lambda_K]$ is finite. Moreover, $\Lambda_F=\Lambda_{F/K}^t$ if and only if $F/K$ is a CM extension.
\end{proposition} 
\begin{proof}
    Since $H$ acts trivially on $\Lambda_K$ by definition, we have $\Lambda_K\subseteq \Lambda_{F/K}^t$. To show that the index $[\Lambda_{F/K}^t:\Lambda_K]$ is finite, we prove $[F:K]\,\Lambda_{F/K}^t\subseteq \Lambda_K$. Let $u\in \Log_F^{-1}(\Lambda_{F/K}^t)$. Using the fact that $H$ acts trivially on $\Lambda_{F/K}^t$, we obtain
    \[
    [F:K]\, \Log_F(u)=\sum_{g\in H}g \Log_F(u)=\Log_F(\N_{F/K}(u))\in \Lambda_K,
    \]
    which gives $[F:K]\, \Lambda_{F/K}^t\subseteq \Lambda_K$ as desired. 

    For the last part, note that since $[\Lambda_{F/K}^t:\Lambda_K]$ is finite, the index $[\Lambda_F:\Lambda_K]$ is finite if and only if $[\Lambda_F:\Lambda_{F/K}^t]$ is finite. By \cite[Corollary 1]{ElementaryandAnalyticofANT}, the former holds precisely when $F/K$ is a CM extension, so it remains to show that $[\Lambda_F:\Lambda_{F/K}^t]$ is finite if and only if $\Lambda_F = \Lambda_{F/K}^t$.

    Since $\Lambda_{F/K}^t = \Lambda_F \cap V_{F/K}^t$, finiteness of $[\Lambda_F:\Lambda_{F/K}^t]$ is equivalent to $V_F = V_{F/K}^t$, which is in turn equivalent to $\Lambda_F = \Lambda_F \cap V_{F/K}^t = \Lambda_{F/K}^t$. The result now follows.

 \end{proof}

\section{Failure of the unit shape as a classifying invariant for CM fields}\label{section: CM-Fields}
It is natural to ask whether the unit shape of a number field determines the field up to isomorphism. In this section, we answer this question in the negative for CM fields; that is, CM fields with the same shape need not be isomorphic. In fact, we show in \Cref{prop: the shape is weak} that for a fixed totally real number field $K$, for all but finitely many CM extensions $F/K$, we have $[\Lambda_F]=[\Lambda_K]$. In \Cref{prop:CMfields} we provide a concrete family of CM extensions $F/K$ for which $[\Lambda_F]=[\Lambda_K]$.
We will see in \Cref{subsection: shapecompleteinvariant} that this is in complete contrast to the setting of totally imaginary non-CM $D_6$-sextic number fields, where shape is a complete invariant. Indeed, \Cref{theorem: Completeness of the shape} establishes that any two non-CM $D_6$-sextic fields $F$ and $F'$ are isomorphic if and only if $[\Lambda_F]=[\Lambda_{F'}]$.

We begin with a useful observation on the unit shapes of CM extensions $F/K$. Recall that there exists a unique injective group homomorphism $\iota: \Lambda_K \longrightarrow \Lambda_F$ that makes the diagram in \Cref{remark: sublattice vs lattice} commute. \Cref{prop: CMsublattice vs lattice} shows that if $F/K$ is a CM-extension, then $[\Lambda_K]=[\iota(\Lambda_K)]$. 

\begin{proposition}\label{prop: CMsublattice vs lattice}
    Let $F/K$ be a CM extension of number fields and let $\Lambda_K$ be the log unit lattice of $K$. Then $[\Lambda_K]=[\Log_F(O_K^\times)]$. 
\end{proposition}
\begin{proof}
    Every (real) embedding $\sigma_i:K\rightarrow \R$ extends to exactly one conjugate pair of (complex) embeddings $\tau_i,\overline{\tau}_i :F\rightarrow \C$.
    Letting  $\iota:\Lambda_K\rightarrow \Log_F(O_K^\times)$ be the group isomorphism of \Cref{remark: sublattice vs lattice}, it is easy to verify that
    \[\langle \Log_F(u), \Log_F(u')\rangle = \langle \iota(\Log_K(u)),\iota(\Log_K(u'))\rangle = 4\langle \Log_K(u),\Log_K(u')\rangle\]
    for all $u, u' \in O_K^\times$. It follows that the homomorphism of groups 
    \[\frac{\iota}{2}:\Lambda_K\rightarrow \frac{\Log_F(O_K^\times)}{2}\] 
    is an isometry, so $[\Lambda_K]=[\Log_F(O_K^\times)]$ as claimed. 
\end{proof}

The unit shape of a CM field is almost always solely determined by the unit shape of its index 2 subfield.

\begin{proposition}\label{prop: the shape is weak}
    Let $K$ be a totally real number field. There exist only finitely many CM extensions $F/K$ such that $[\Lambda_K] \neq [\Lambda_F]$. 
\end{proposition}
\begin{proof}
    By \Cref{prop: CMsublattice vs lattice}, we have $[\Lambda_K] = [\Log_F(O_K^\times)]$, so we will identify these two lattices and view $\Lambda_K$ as a sublattice (and not just a subgroup) of $\Lambda_F$. A classical result of Remak~\cite{remak1954algebraische} establishes the inequality $[\Lambda_F:\Lambda_K]\leq 2$. It follows that if the inclusion $O_K^\times\hookrightarrow O_F^\times$ does not induce an isomorphism of abelian groups between $E_K$ and $E_F$, then $[\Lambda_F:\Lambda_K]=2$. Let $v\in \Lambda_F\setminus \Lambda_K$, and let $\epsilon\in F$ be a unit such that $\Log_F(\epsilon)=v$. Since $[F:K]=2$, we have $F=K(\epsilon)$. By \cite[Theorem 1]{greither2017cm}, there are only finitely many CM extensions $F/K$ of the form $F=K(\epsilon')$ for some unit $\epsilon'\in O_F^\times$.  This gives the result.
\end{proof}

Next, we give an explicit family of CM fields whose unit shapes are always the same as the unit shape of their common index 2 subfield. The proof of this result makes use of the following standard result about rings of integers of composite fields.

\begin{lemma}[\mbox{\cite[Prop.\ 2.11 and the subsequent remark]{Neukirch}}]\label{prop:index1}
    Let $F$ and $F'$ be number fields. 
    If $F$ and $F'$ have coprime discriminants, then $O_{FF'} = O_F \otimes_\Z O_{F'}$.
  \end{lemma}

\begin{proposition}\label{prop:CMfields}
    Let $K$ be a totally real number field, $F=K(\sqrt{-d})$ with $d\in \Z_{>0}$ squarefree, and $F_2 = \Q(\sqrt{-d})$. Let $\Delta_K$ and $\Delta_{F_2}$ denote the discriminants of $K$ and $F_2$, respectively. Suppose that
    \begin{itemize} 
        \item  $\mu_F$ has order at most $6$ and 
        \item either $d\not\equiv 1,3\pmod{(K^\times)^2}$ and $[K:\Q]$ is odd, or $d \in \{ 1, 3\}$ and $\gcd(\Delta_K, \Delta_{F_2}) = 1$.  
    \end{itemize}
    Then $[\Lambda_F]=[\Lambda_K]$.  
\end{proposition}

\begin{proof}
Let $k\coloneq[K:\Q]$. Since $F/K$ is a CM extension, we may assume that $\Lambda_K$ is a sublattice of $\Lambda_F$ by \Cref{prop: CMsublattice vs lattice}. Then 
    $\Lambda_{F}$ and $\Lambda_{K}$ both have rank $k-1$. It follows that the index $\ell=[\Lambda_F:\Lambda_K]$ is finite. 
       
    Assume to the contrary that $\Lambda_{K}$ is properly contained in $\Lambda_{F}$. Then there exists $v\in \Lambda_F \setminus \Lambda_{K}$. Note that $\ell v \in \Lambda_K$. Let $\rho$ be the generator of $H = \Gal(F/K)$ and consider $\rho$ as an element in the group ring $\Z[H]$ acting on $\Lambda_F$. Then $\ell v = \rho \cdot \ell v = \ell \rho\cdot v$, so $v = \rho \cdot v$. Let $\epsilon \in O_F^\times$ such that $\Log_F(\epsilon) = v$.   
    Then $\rho \cdot v = v$ implies that $\epsilon^\rho =\zeta \epsilon$ for some $\zeta\in \mu_F$.
Note that $K \cap F_2 = \Q$ and $F = KF_2$. Let $t = [O_F: O_K[\sqrt{-d}]]$. Then we can write $\epsilon$ in the form $\epsilon= \frac{1}{t}(\theta_1 + \theta_2 \sqrt{-d})$, with $\theta_1, \theta_2 \in O_K$. Computing the relative norm of $\epsilon$ yields
\[ \N_{F/K}(\epsilon) = \epsilon^{1+\rho} = \zeta \epsilon^2 = \frac{\zeta}{t^2} \left ( \theta_1^2-d\theta_2^2+2\theta_1\theta_2\sqrt{-d} \right ) \in O_K^\times. \]
 By assumption $F$ contains 2, 4 or 6 roots of unity, so $\zeta \in \{\pm 1, \pm i, \pm\zeta_3, \pm\zeta_3^2\}$ where $i$ is a primitive fourth root of unity and $\zeta_3$ is a primitive cube root of unity. 

\textbf{Case 1.} $d \not\equiv 1,3\pmod{(K^\times)^2}$ and $k$ is odd. Then $\sqrt{-1} \notin O_F$ and $\sqrt{-3} \notin O_F$, so $\pm i, \pm \zeta_3, \pm \zeta_3^2 \notin \mu_F$, forcing $\zeta = \pm 1$. Since $\N_{F/K}(\epsilon) \in O_K$, we must have $\theta_1 = 0$ or $\theta_2 = 0$. 

If $\theta_1 = 0$, then $\epsilon= \theta_2 \sqrt{-d}/t$, so $\N_{F/K}(\epsilon) = d \theta_2^2/t^2 \in O_{K}^{\times}$. Computing norms  down to $\mathbb{Q}$ yields
$$\pm 1 = \N_{F/\mathbb{Q}}(\epsilon) = \N_{K/\mathbb{Q}}(\N_{F/K}(\epsilon)) = \frac{d^k \N_{K/\mathbb{Q}}(\theta_2)^2}{t^{2k}}, $$
so $\pm d^k$ is a square in $\mathbb{Q}$ and hence in $\mathbb{Z}$. Since $k$ is odd and $d$ is squarefree, this is a contradiction.
If $\theta_2=0$, then $\epsilon = \theta_1/t \in K \cap O_F^\times = O_K^\times$. But then $v \in \Lambda_K$, which is again a contradiction.

\textbf{Case 2.} $d = 1$ and $\gcd(\Delta_K, \Delta_{F_2}) = 1$. Then $O_F = O_K \otimes O_{F_2}$ by \Cref{prop:index1}, so $t = 1$. Here, $\zeta = i^s$ for some   $s \in \{0, 1, 2, 3\}$, so  
    \[ \N_{F/K}(\epsilon) = i^s(\theta_1^2 - \theta_2^2 + 2\theta_1\theta_2 i) .\]
    If $s \in \{0, 2 \}$, then $\theta_1 = 0$ or $\theta_2 = 0$, so $\epsilon = \theta_1$ or $\epsilon = \theta_2 i$. In either case $\epsilon \in O_{K}^\times \times \mu_F$, so $v \in \Lambda_{K}$, which contradicts our original assumption. If $s \in \{1, 3\}$, then $\theta_1^2 = \theta_2^2$, so $\epsilon = \theta_1(1 \pm i)$. But then $\N_{F/\mathbb{Q}}(1 \pm i) = 2^k$ divides $\N_{F/\mathbb{Q}}(\epsilon) = \pm 1$ which is impossible. 

\textbf{Case 3.} $d = 3$ and $\gcd(\Delta_K, \Delta_{F_2}) = 1$. 
Then again $O_F = O_K \otimes O_{F_2}$ by \Cref{prop:index1}. Here, $\zeta = \pm \zeta_3^s$ for some $s \in \{0, 1, 2\}$, so we can write $\epsilon = \theta_1 + \theta_2 \zeta_3$ with $\theta_1, \theta_2 \in O_{K}$. Then  
\[
     \N_{F/K}(\epsilon) = \pm\zeta_3^s \big ( ( \theta_1^2 - \theta_2^2)  + \zeta_3(2\theta_1\theta_2 - \theta_2^2) \big ) = \pm\zeta_3^s \big ( (\theta_1^2 - 2\theta_1\theta_2) + \zeta_3^2(\theta_2^2 - 2\theta_1\theta_2) \big ).  \]
If $s = 0$, then $2\theta_1\theta_2 - \theta_2^2 = 0$, so $\theta_2 = 0$ or $\theta_2 = 2\theta_1$. In the first case $\epsilon = \theta_1 \in O_{K}^\times$, so $v \in \Lambda_{K}$ which is a contradiction. In the second case, $\epsilon = \theta_1(1+2\zeta_3)$, in which case $\N_{F/\mathbb{Q}}(1-\zeta_3^2) = 3^k$ divides $\N_{F/\mathbb{Q}}(\epsilon) = \pm 1$ which is impossible. 

If $s = 1$, then $\theta_1^2 - 2\theta_1\theta_2 = 0$, so $\theta_1 = 0$ or $\theta_1 = 2\theta_2$. In the first case, $\epsilon = \theta_2 \zeta_3 \in O_{K}^\times\times \mu_F$, so $v \in \Lambda_{K}$, a contradiction. In the second case, $\epsilon = \theta_2(2+\zeta_3)$, which again leads to the contradiction of $3^k$ having to divide $\pm 1$. 

If $s = 2$, then $\theta_1 = \pm \theta_2$, in which case $\epsilon = \theta_1(1 \pm \zeta_3)$. In the plus case, $\epsilon  = -\theta_1\zeta_3^2 \in \langle O_{K}^\times, \mu_F\rangle$, so $v \in \Lambda_{K}$, a contradiction. In the minus case, $\N_{F/\mathbb{Q}}(1-\zeta_3) = 3^k$ divides $\N_{F/\mathbb{Q}}(\epsilon) = \pm 1$ which is again impossible.
\end{proof}

\section{Rank 2 log unit lattices} \label{sec:rank2}

An important question pertaining to log unit lattices asks where the unit shapes of number fields are located within the moduli space of all lattices of the same rank, and whether families of fields are algebraically constrained to special subsets of that space. In this section we give an almost complete answer in the rank 2 case, i.e.\ for number fields of signature $(r,s)$ with $r+s-1=2$ except those specified in \Cref{conj:remaining Galois groups}. We first enumerate all possible Galois groups of the Galois closure  $\tilde{F}$ of such a field $F$ (\Cref{prop:Galois groups rank 2}), which forces the degree $[F:\Q]$ to be 3, 4, 5, or 6 depending on the signature. 

Our main result of this section, Theorem~\ref{theorem: classification or rank 2}, organizes the families of unit rank 2 fields characterized in \Cref{prop:Galois groups rank 2} into three sharply distinct regimes. At one extreme, fields whose Galois group carries an automorphism of order 3 acting non-trivially on the unit lattice have hexagonal unit shape. At the other extreme, recalling the notation introduced at the beginning of   \Cref{ss:unit-rank-classification}, families such as $\Omega(S_3, 3, 0)$ and $\Omega_{\mathrm{CM}}(G,0,3)$ have unit shapes that, conditionally on the Weak Schanuel Conjecture (\Cref{conj: algebraic independence of logs}), avoid every proper $\overline{\mathbb{Q}}$-algebraic subvariety of $\mathcal{S}_2$, so meaning they are as algebraically unconstrained as possible. In between lie families such as $\Omega(D_4,2,1)$ and $\Omega_{\mathrm{non\text{-}CM}}(G,0,3)$, whose unit shapes are transcendental but are confined to the boundary of $\mathcal{S}_2$. This is particularly striking in contrast with the $D_5$ family studied in \cite{HHV}, whose unit shapes are constrained to an explicit algebraic curve in $\mathcal{S}_2$; our results show this additional algebraic rigidity is not a general feature of rank 2 unit shapes. The cases of fields with no proper subfields, and of $S_3^2$- and $S_3^2\rtimes C_2$-sextic fields, remain open; computational evidence supports the conjecture that their unit shapes are transcendental and lie in the interior of $\mathcal{S}_2$ (Conjecture~\ref{conj:remaining Galois groups}).

Henceforth, let $\mathcal{F}(G,r,s)$ denote the set of number fields $F$ with signature $(r,s)$ whose Galois closure $\Tilde{F}$ has Galois group $\Gal(\tilde{F}/\Q) = G$. We partition this set into the subsets $\mathcal{F}_{\CM}(G, r, s)$ and $\mathcal{F}_{\non-CM}(G, r, s)$ consisting of fields in $\mathcal{F}(G, r, s)$ with and without complex multiplication, respectively. 


\subsection{Possible Galois groups in unit rank 2}

\begin{proposition}\label{prop:Galois groups rank 2}
   The following isomorphism classes of Galois groups $G$ are possible for fields $F \in \mathcal{F}(G,r,s)$ with unit rank $r+s-1 = 2$, corresponding to each choice of signature $(r,s)$ of $F$:
    \begin{enumerate}
   \item[(i)] If $(r,s)=(3,0)$, then $[F:\Q] = 3$ and $G\cong C_3$ or $S_3$. 
    \item[(ii)] If $(r,s)=(2,1)$, then $[F:\Q] = 4$ and $G\cong D_4$ or $S_4$. In the latter case, $F$ does not contain any non-trivial subfields. 
    \item[(iii)] If $(r,s)=(1,2)$, then $[F:\Q] = 5$ and $G\cong S_5,\: A_5,\: D_5$, or the Frobenius group $F_5 = C_5\rtimes C_4$ (labeled 5T3 in \cite{lmfdb}).
    \item[(iv)] If $(r,s)=(0,3)$, then $[F:\Q] = 6$ and the following Galois groups are possible:
    \begin{enumerate}
        \item[(a)] If $F$ contains a quadratic subfield $F_2$, then $G\cong C_6$, $S_3$, $D_6$, $S_3\times C_3$, $S_3^2$, or $S_3^2 \rtimes C_2$ (labeled 6T13 in \cite{lmfdb}). Moreover, the extension $F/F_2$ is Galois if and only if $G \cong C_6$, $S_3$, or $S_3\times C_3$. 
        \item[(b)] If $F$ contains a cubic subfield, then  $G \cong C_6$, $S_3$, $D_6$, $A_4\times C_2$, $S_4$, or $S_4\times C_2$.
        \item[(c)] If $F$ contains no non-trivial subfields, then $G \cong S_5$ or $S_6$.
    \end{enumerate}
\end{enumerate}
\end{proposition}
\begin{proof} Recall if $F$ has degree $n$, then the Galois group $\Gal(\Tilde{F}/\Q)$ of the Galois closure $\Tilde{F}$ of $F$ embeds into $S_n$ as a transitive subgroup. Also note that the embeddings of a normal extension over $\Q$ are either all totally real or all complex.
\begin{enumerate}
    \item[(i)] This case is well-known.
    
    \item[(ii)]  The transitive subgroups of $S_4$ are $C_4$, $C_2^2$, $D_4$, $A_4$, and $S_4$. In this case $C_4$, $C_2^2$ and $A_4$ are excluded. The $C_4$ and $C_2^2$ cases are excluded because they correspond to normal quartic fields of signature $(4,0)$ or $(0,2)$.
    
    The group $A_4$ cannot be the Galois group of a quartic with signature $(2,1)$. To see this, note that $\tilde{F}$ has degree $|A_4| = 12$ in this case. Let $F_R$ be the index~2 subfield of $\tilde{F}$ fixed by complex conjugation.Then $F$ must have an embedding into $F_R$ since $F$ has a real embedding. However, $[F:\Q] = 4$ does not divide $[F_R:\Q] = 6$, so this is imposible. 
    
    The remaining groups, $D_4$ and $S_4$, appear as the Galois group of the Galois closure of a quartic with signature $(2,1)$. See \cite{lmfdb} for examples of number fields with these Galois groups.
    
    \item[(iii)] The transitive subgroups of $S_5$ are $C_5$, $D_{5}$, $F_5$, $A_5$, and $S_5$. Among these, only $C_5$ is excluded because it implies $F=\Tilde{F}$, in which case $F$ has signature $(5,0)$. The remaining groups appear as the Galois group of the Galois closure of a quintic with signature $(1,2)$. See \cite{lmfdb} for examples of number fields with these Galois groups.
    
    \item[(iv)] The isomorphism classes of the transitive subgroups of $S_6$ are $C_6$, $S_3$, $D_{6}$, $A_4$, $S_3\times C_3$, $A_4\times C_2$, $S_4$, $S_3^2$, $C_3^2\rtimes C_4$ (labeled 6T10 in \cite{lmfdb}), $S_4\times C_2$, $A_5$, $S_3^2\rtimes C_2$ (labeled 6T13 in \cite{lmfdb}), $S_5$, $A_6$, $S_6$. We split these up according to whether they occur as the Galois group of (a) a sextic field $F$ with a quadratic subfield, (b) a sextic field $F$ with a cubic subfield, and (c) a sextic field $F$ with no non-trivial subfields:
\begin{itemize}
    \item[(a)] $C_6$, $S_3$, $D_6$, $S_3\times C_3$, $S_3^2$, $C_3^2\rtimes C_4$, and $S_3^2\rtimes C_2$.
    \item[(b)] $C_6$, $S_3$, $A_4$, $D_6$, $S_4$, $A_4\times C_2$, and $S_4\times C_2$.
    \item[(c)] $A_5$, $S_5$, $A_6$, and $S_6$.
\end{itemize}
This matches the remark in \cite[pp 325]{cohen1993course}, although note that there are two ways of embedding $S_4$ as a transitive subgroup of $S_6$. These are labeled 6T7 and 6T8 in \cite{lmfdb}, respectively. By \cite[Theorem 2]{WreSmit}, we can find possible Galois groups of the Galois closure of a sextic with a quadratic subfield or a cubic subfield by looking at the transitive subgroups of the wreath products $S_3\wr C_2$ and $C_2\wr S_3$\,, respectively. Note that $S_3\wr C_2\cong S_3^2\rtimes C_2$ and $C_2\wr S_3\cong S_4\times C_2$; this fact is stated in \cite{lmfdb}. Among the subgroups of $S_3\wr C_2$ and $C_2\wr S_3$, the only groups excluded are those without an element that is the product of three disjoint 2-cycles. This is because complex conjugation must act non-trivially on the roots of a sextic with signature $(0,3)$. Furthermore, since all the roots are complex and complex roots always come in pairs, the automorphism of complex conjugation must embed into $S_6$ as the product of three disjoint 2-cycles. The groups which do not have any embeddings as a transitive subgroup of $S_6$ containing a product of three disjoint 2-cycles are the transitive embedding of $S_4$ that is labeled 6T7,  $A_4$, $C_3^2\rtimes C_4$, $A_5$, and $A_6$. Note that the remaining groups do appear as the Galois group of the Galois closure of a totally imaginary sextic. One can find examples of these sextics in \cite{lmfdb}. Thus, our final list is:
\begin{itemize}
    \item[(a)] $C_6$, $S_3$, $D_6$, $S_3\times C_3$, $S_3^2$, and $S_3^2\rtimes C_2$.
    \item[(b)] $C_6$, $S_3$, $D_6$, $A_4\times C_2$, $S_4$ and $S_4\times C_2$. 
    \item[(c)] $S_5$ and $S_6$.
\end{itemize}
\end{enumerate}
This concludes the result.
\end{proof}

\begin{remark}\label{rmk: rank 2 Galois number fields}
  \Cref{prop:Galois groups rank 2} shows that a number field $F$ with unit rank 2 is Galois if only if $\Gal(F/\Q)$ is isomorphic to $C_3, C_6$, or $S_3$.
\end{remark}

\begin{definition}
    We denote the boundary of $\mathcal{S}_2$, viewed as a subset of $\C$, by $\partial\mathcal{S}_2$, so 
\[
 \partial \mathcal S_2=\{x+iy\in\mathcal{S}_2\, :\, x=0\}\cup {\{x+iy\in\mathcal{S}_2\, :\, x^2+y^2=1\}}\cup \{x+iy\in\mathcal S_2\, :\, x=1/2 \}.
\]

We refer to the three sets in the union above as the \emph{left boundary}, \emph{lower arc} and \emph{right boundary}, respectively.
 We also call a unit shape  \textit{transcendental} if it is a transcendental number in $\mathcal{S}_2$.
\end{definition}
\begin{remark}\label{remark: geometric properties from fundamental domain}
    Let $\Lambda$ be a rank $2$ lattice and let $z=x+iy\in \mathcal{S}_2$ be the point that represents the shape of $\Lambda$. The description of $\mathcal{S}_2$ given in \Cref{sec:background} shows that $\Lambda$ is orthogonal if and only if $z$ lies on the left boundary of $\mathcal{S}_2$, and $\Lambda$ is well-rounded if and only if $z$ lies on the lower arc. 
\end{remark}

The remainder of this section is devoted to the proof of the following theorem, which is one of the main results of this paper. As usual, let $\overline{\Q}$ denote the field of algebraic numbers.

\begin{theorem}\label{theorem: classification or rank 2}
Let $F \in \mathcal{F}(G,r,s)$, where $r+s-1 = 2$ and $G$ is one of the Galois groups of \Cref{prop:Galois groups rank 2}. Define 
\begin{equation} \label{eq:Galoisgroups}
    \mathcal G =\{D_6, S_4, A_4\times C_2,S_4\times C_2\}.
\end{equation}
   Then the following hold.
    \begin{enumerate}
        \item If $G \cong C_3$, $C_6$, or $S_3$ (so $F/\Q$ is Galois by \Cref{rmk: rank 2 Galois number fields}), or $G\cong S_3\times C_3$, then the log unit lattice of $F$ is hexagonal.

        \item All the unit shapes in $\Omega(D_4,2,1)$ and $\Omega_{{\non-CM}}(G,0,3)$, with $G$ isomorphic to one of the groups in $\mathcal G$, are transcendental and lie on the boundary of $\mathcal{S}_2$. Moreover, for any real quadratic field $F_2$ (resp.\ complex cubic field $F_3$), there are only finitely many fields in $\Omega(D_4,2,1)$ (resp.\ $\Omega_{{\non-CM}}(G,0,3)$) containing $F_2$ (resp.\ $F_3$) whose unit shape lies on the lower arc of $\mathcal{S}_2$.

        \item  Assuming \Cref{conj: algebraic independence of logs}, all unit shapes in $\Omega(S_3,3,0)$ and $\Omega_{\CM}(G,0,3)$, with $G$ isomorphic to one of the groups in $\mathcal G$, lie in the interior of $\mathcal{S}_2$. Moreover, for every complex number $z=x+iy\in\mathcal S_2$ representing one of these shapes, no proper Zariski-closed subset of $\C^2$ defined over~$\overline{\Q}$ contains the point $(x,y)$. In particular, the complex number $x+iy$ is transcendental.    
       
    \end{enumerate}
\end{theorem}
\begin{proof}
    Part 1 is \Cref{prop: hexagonal lattices}, part 2 is obtained in \Cref{cor: unit shapes in the boundary} and \Cref{theorem: trascendental point}, and  part 3 is \Cref{prop: shape_space_totally_real}. 
\end{proof}

Note that \Cref{theorem: classification or rank 2} misses the unit shapes of quartic, quintic, and sextic number fields with no non-trivial subfields, and totally imaginary sextic fields with Galois groups \(S_3^2\) and \(S_3^2\rtimes C_2\). To examine these remaining cases, we used PARI/GP \cite{PARI2} to compute unit shapes from LMFDB data for all families of number fields listed in the following conjecture. To obtain computational evidence for transcendentality, we applied the PSLQ algorithm as implemented in \texttt{PARI/GP}. For the number fields in the LMFDB data set considered here, no polynomial relation of degree at most 6, with all coefficients bounded in absolute value by $10^8$, was detected. This observation motivates the following conjecture.

\begin{conjecture}\label{conj:remaining Galois groups}
Let $F \in \mathcal{F}(G,r,s)$, where $r+s-1 = 2$ and $G$ is one of the Galois groups of \Cref{prop:Galois groups rank 2}. Suppose $F$ satisfies one of the following conditions:
\begin{enumerate} \itemsep 0pt
    \item[(i)] $F$ has no non-trivial subfields and $F$ has signature $(r,s)=(2,1),(1,2)$, or $(0,3)$;
    \item[(ii)] $G \cong S_3^2$ or $S_3^2\rtimes C_2$ (so $(r,s) = (0,3)$ by \Cref{prop:Galois groups rank 2}).
\end{enumerate}  
Then the unit shape of $F$ is in the interior of $\mathcal{S}_2$ and is transcendental.
\end{conjecture}

\begin{remark}\label{remark: comparison of results}
     As mentioned in \Cref{subsection: previous work}, parts of \Cref{theorem: classification or rank 2} have already been observed in prior literature. 
  
     The fact that the unit shapes in $\Omega(C_3,3,0)$ and $\Omega(C_6,0,3)$ are hexagonal was already established in \cite{tran2023size}, \cite{TranTian2018SizeFunction}, \cite{TPS2021}, \cite{terr1997}. In fact, our proof of part 2 of \Cref{theorem: classification or rank 2} follows a similar strategy to the one given in \cite{RNT} for the $D_4$ case. 
     Carrying out an analogous computation using the techniques of \cite{HHV} in the case of totally imaginary $D_3$ ($\cong S_3$) shows that the shapes of these rank 2 unit lattices are hexagonal.
    
     It is interesting to compare part 3 of \Cref{theorem: classification or rank 2} with the results of \cite{RNT} and a particular instance of one of the results in \cite{HHV}. In \cite{RNT}, it was already shown that the unit shapes in $\Omega(D_4,2,1)$ lie on the boundary of $\mathcal{S}_2$ (i.e.\ a curve) and are transcendental \cite{RNT}. In \cite{HHV}, the authors show that any unit shape in $\Omega(D_5,1,2)$ can be represented by a complex number $x+iy$ on the curve
     \[
     \left(x+\frac{1}{2}\right)^2+\left(y-\frac{1}{2\sqrt{3}}\right)^2=\left(\frac{2}{3}\right)^2.
     \]
     Note the disparity of these cases with the behavior of the unit shapes in part 3 of \Cref{theorem: classification or rank 2} which, assuming \Cref{conj: algebraic independence of logs}, are not contained in any $\overline{\Q}$-algebraic subset of~$\C^2$. 
\end{remark}

The remainder of this subsection is devoted to the proof of part 1 of \Cref{theorem: classification or rank 2}; the other parts will be proved in subsequent subsections. As mentioned in \Cref{remark: comparison of results}, this result is already known for the cases where $F/\Q$ is Galois and has unit rank 2, and $\Gal(F/\Q)$ is isomorphic to $C_3$, $C_6$, or $S_3$.

\begin{proposition}[\Cref{theorem: classification or rank 2} part 1]\label{prop: hexagonal lattices}
    Let $F \in \mathcal{F}(G,r,s)$ with $r+s-1 = 2$, where either $G \cong C_3$, $C_6$, or $S_3$ (so $F$ is Galois by \Cref{remark: comparison of results}), or $G \cong S_3 \times C_3$. Then $\Lambda_F$ is the hexagonal lattice.
\end{proposition}
\begin{proof}
    
    By \cite[Proposition 2.1]{TranTian2018SizeFunction}, it suffices to prove that $\Lambda_F$ admits an automorphism of order 3 (see also \cite[Table 6.2]{terr1997}). If $F/\Q$ is Galois, then for all three choices, $G$ admits an automorphism of order 3 that does not act trivially on $\Lambda_F$ by \Cref{prop: subfield in trivial part}. If $G \cong S_3 \times C_3$, then $F$ is a sextic field that contains a unique quadratic subfield $F_2\subseteq F$ such that $F/F_2$ is Galois. It follows, just as before, that $\Gal(F/F_2)$ has an automorphism of order 3 which does not act trivially on $\Lambda_F$ by  \Cref{prop: subfield in trivial part}. 
    \end{proof}

\subsection{Transcendental unit shapes on the boundary}
We now proceed with the proof of part 2 of \Cref{theorem: classification or rank 2} which will be divided into two parts. In \Cref{cor: unit shapes in the boundary}, we prove that $\Omega(D_4,2,1)$ and $\Omega_{\non-CM}(G,0,3)$ for $G\in \mathcal G$ as in \eqref{eq:Galoisgroups}, are contained in the boundary of $\mathcal{S}_2$. In \Cref{theorem: trascendental point}, we establish that these sets consist of transcendental numbers. In addition, \Cref{cor:unit-basis} generalizes Nakamula's description \cite[Corollary 2]{nakamula1980group} of the fundamental units of fields in $\mathcal{F}(D_4,2,1)$ and $\mathcal{F}_{\non-CM}(D_6,0,3)$ to the families $\mathcal{F}_{\non-CM}(G,0,3)$ with $G\in\mathcal{G}$.

As a consequence of the above results, \Cref{cor: characterization of orthogonality} gives necessary and sufficient conditions for a lattice whose unit shape is transcendental and lies on the boundary of $\mathcal{S}_2$ to be orthogonal. This will be used in \Cref{sec:geomD6}.

\subsubsection{Fundamental units and successive minima}

Our investigation will require explicit bases for the log unit lattices of fields in the families $\mathcal{F}(D_4,2,1)$ and $\mathcal{F}_{\non-CM}(G,0,3)$, with $G\in \mathcal{G}$ as given in \eqref{eq:Galoisgroups}. In the process, we also determine a pair of fundamental units for these fields. 

\begin{proposition}\label{prop: boundarytype}
    Every $F\in \mathcal{F}{(D_4,2,1)} \cup \mathcal{F}_{\non-CM}(G,0,3)$, with $G\in \mathcal{G}$ as in \eqref{eq:Galoisgroups}, has a subfield $K$ of index $2$. Moreover, $\rk(\Lambda_K)=\rk(\Lambda_{F/K}^t)=\rk(\Lambda_{F/K}^r)=1$, where $\Lambda_{F/K}^t$ and $\Lambda_{F/K}^r$ are defined in \eqref{eq:orthorep}.
\end{proposition}
\begin{proof}
   Every field $F \in \mathcal{F}{(D_4,2,1)}$ is quartic and contains a real quadratic subfield $K = F_2$, so $\rk(\Lambda_K) = 1$. Furthermore, \cite{RNT} establishes that $\rk(\Lambda_{F/K}^t)=\rk(\Lambda_{F/K}^r)=1$.
   
    For each of the groups $G \in \mathcal{G}$, any field $F \in \mathcal{F}_{\non-CM}(G,0,3)$ contains a cubic subfield $K = F_3$ which is not totally real since $F$ is not a CM field. \Cref{prop: subfield in trivial part} now implies that 
    $\rk(\Lambda_K)=\rk(\Lambda_{F/K}^t)=1$, which forces $\rk(\Lambda_{F/K}^r)=1$. \end{proof}

For the remainder of this section, let $F/K$ be as in \Cref{prop: boundarytype}, and write $H\coloneq\Gal(F/K) = \langle \gamma \rangle \cong C_2$. Then $\Lambda_{F/K}^t$ and $\Lambda_{F/K}^r$ have rank~1. 
We introduce the following three vectors $v_t, v_r, v$ that will play a crucial role in establishing that the unit shape of $F$ lies on the boundary of $\mathcal{S}_2$.
\begin{equation} \label{eq:vectors}
    \Lambda_{F/K}^t = \langle v_t \rangle, \qquad \Lambda_{F/K}^r = \langle v_r \rangle, \qquad v = \frac{v_t+v_r}{2}  \in \Q\otimes_\Z\Lambda_F. 
\end{equation}
Specifically, $v_t$ and $v_r$ are generators of the respective abelian groups $\Lambda_{F/K}^t$ and $\Lambda_{F/K}^r$. Note that $v$ need not belong to $\Lambda_F$.

\begin{remark}\label{remark: definition of vr and vt}
     From the definition of $v_t$ and $v_r$ it follows that the following maps  are isomorphisms of $\Z[H]$-modules:
    \begin{align*}
         \Z[\gamma]/&(\gamma-1)\rightarrow \Lambda_F^t,
         &&\text{and}&\Z[\gamma]/&(\gamma+1)\rightarrow\Lambda_F^r,\\
         &1\mapsto v_t,&&&&1\mapsto v_r.
    \end{align*}
   For any $w\in\Lambda_F$, we have $(1+\gamma)w\in\Lambda^t$ and $(1-\gamma)w\in\Lambda^r$. In fact, for any $u\in E_F$, we have $u^{1+\gamma}=\N_{F/K}(u)$.  
\end{remark}

We can now describe the possible $\Z[H]$-module structures of $\Lambda_F$.

\begin{proposition}\label{prop: unit structure}
       Let $F/K$ be as in \Cref{prop: boundarytype}, and let $v_t, v_r, v$ be as in \eqref{eq:vectors}. For brevity, write $\Lambda=\Lambda_F$, $\Lambda^t=\Lambda_{F/K}^t$ and $\Lambda^r=\Lambda_{F/K}^r$. Define $\Lambda'= \Lambda^t\oplus \Lambda^r$ and $A=\mathbb Z[\gamma]/(\gamma^2-1)$. Then the following hold.
        
        \begin{enumerate} 
        \item[(i)] $\Lambda=\Lambda'$ or 
        \item[(ii)] $\Lambda\cong A$ which is equivalent to $v \in\Lambda$. Furthermore, $[\Lambda:\Lambda']=2$ and any pre-image $u\in \Log_{F}^{-1}(\{v\})$ is a Minkowski unit of $F$.
        \end{enumerate}
\end{proposition}
\begin{proof}
 As in the proof of \Cref{lem:vtvrortho}, we view the norm map $\N_{F/K} = \gamma+1$ as an element in $\Z[H]$. Let $\psi:\Lambda\rightarrow \Lambda^t/2\Lambda^t$ be the following composition of maps, where the second map is the natural projection:
 \[
 \psi:\Lambda\xrightarrow{\N_{F/K}}\Lambda^t\xrightarrow{\mathrm{proj}}\Lambda^t/2\Lambda^t.
 \]
 We claim that $\ker(\psi)=\Lambda'.$ The inclusion $\Lambda'\subseteq \ker(\psi)$ is not hard to verify, so we only prove that $\ker(\psi)\subseteq \Lambda'$. Let $w\in \ker(\psi)$. Then $\N_{F/K}\,w=2w'$ for some $w'\in \Lambda^t$. It follows that $w-w'\in \ker(\N_{F/K})=\Lambda^r$; hence, by \Cref{lem:vtvrortho}, $w\in \Lambda'$. We conclude that $[\Lambda:\Lambda']=\text{1 or 2}$.
 
If $[\Lambda:\Lambda']=1$, then the proof of part (i) is complete, so assume that $[\Lambda:\Lambda']=2$. We show that $v \in \Lambda \setminus \Lambda'$. To that end, let $w \in \Lambda \setminus \Lambda'$. Then $\N_{F/K}\,w \in \Lambda^t$ and $(1-\gamma)w \in \Lambda^r$ by \Cref{remark: definition of vr and vt}. Write 
\[
w=pv_t+qv_r,\quad p,q\in \Q.
\]
Since $w \notin \Lambda' = \ker(\psi)$ and $\N_{F/K} \, v_t = 2v_t$, we have  $2pv_t = \N_{F/K}w \in \Lambda^t \setminus 2\Lambda^t$, so $p \in \frac{1}{2}\Z \setminus \Z$. Similarly, since $(1-\gamma)\,v_r = 2v_r$, we have $2qv_r = (1-\gamma)\,w \in \Lambda^r$, so $q \in \frac{1}{2}\Z \setminus \Z$. It follows that 
$w-v=(p-1/2)v_t+(q-1/2)v_r \in \Lambda'$, and hence $v \in \Lambda \setminus \Lambda'$.

 A straightforward computation shows that $\mathrm{span}_A(\{v\})$ contains $\Lambda'$ as a proper submodule, so $\Lambda=\mathrm{span}_A(\{v\})$. It follows that the homomorphism of $A$-modules
 \[
 A\rightarrow \Lambda,\quad \alpha\mapsto \alpha v
 \]
 is surjective. Since $A$ and $\Lambda$ are free abelian groups of the same rank, any surjection from~$A$ to~$\Lambda$ is an isomorphism, so $\Lambda \cong A$ as $A$-modules. 

 To complete part (ii), observe that we already proved that if $[\Lambda : \Lambda'] = 2$, then $\Lambda \cong A$ and $v \in \Lambda$. For the converse, note that if $v \in \Lambda$, then $\mathrm{span}_A(\{v\})$ contains $\Lambda'$ as a proper submodule, so $[\Lambda:\Lambda']=2$. Part (i) now forces $\Lambda\cong A$ which yields the result.
\end{proof}

The algebraic structures of \Cref{prop: unit structure} are intimately related to the geometry of the lattice $\Lambda_F$. In fact, they almost determine its shortest vectors, as the following corollary shows.

\begin{corollary}\label{cor: reduced basis}
    Let $A$ and $\Lambda$ be as in  \Cref{prop: unit structure} and $v$, $v_t$, $v_r$~as given in~\eqref{eq:vectors}. Then one of the following holds:
    \begin{enumerate}
        \item[(i)] $\Lambda$ is orthogonal and the set         $\{v_t,v_r\}$ is an orthogonal basis of $\Lambda$         that achieves the successive minima of $\Lambda$.  
        \item[(ii)] $\Lambda\cong A$, $v\in\Lambda$, and one of the sets $\{v_t,v\}$, $\{v_r,v\}$, $\{v,\gamma v\}$ is a basis that achieves the successive minima of $\Lambda$. Moreover, in this case, $\Lambda$ is well-rounded if and only if $v$ is a shortest vector of $\Lambda$.
    \end{enumerate}
\end{corollary}
\begin{proof}
    If $\Lambda=\Lambda^t\oplus\Lambda^r$, then by \Cref{lem:vtvrortho}, (i) must hold. Otherwise, $\Lambda\cong A$ and $v=(v_t+v_r)/2\in \Lambda$ by \Cref{prop: unit structure}, and it is clear that all three sets of vectors in (ii) form a basis of $\Lambda$. 
    
    If $v$ is a shortest vector of $\Lambda$, then since $\gamma v$ and $v$ have the same length, they achieve the successive minima of $\Lambda$. If $v$ is not shortest, then let $w$ be the shorter vector among $v_t$ and $v_r$. Then \cite[Lemma 17.1.4]{galbraith2012mathematics} implies that the basis $\{w,v\}$ achieves the successive minima of $\Lambda$. This completes the classification of the successive minima of $\Lambda$.

    This shows that if $v\in\Lambda$, then $v$ achieves one of the successive minima of $\Lambda$. It follows that if $v\in\Lambda$, then $\Lambda$ is well-rounded if and only if $v$ is a shortest vector of $\Lambda$, which proves the second assertion in (ii).  
\end{proof}

The following result, for the special cases of $D_4$-quartic fields of signature $(2,1)$ and for totally imaginary non-CM $D_6$-sextic number fields, is part of \cite[Corollary 2]{nakamula1980group}.

\begin{corollary}\label{cor:unit-basis}
    Let $F/K$ be as in \Cref{prop: boundarytype}. Let $\epsilon_t\in \Lambda_{K}$ be a fundamental unit of~$K$ and $\epsilon_r\in O_F^\times$ be such that $v_r=\Log_F (\epsilon_r)$ is a generator of $\Lambda^r$. Then exactly one of the following holds: 
    \begin{enumerate}
        \item[(i)] The set $\{\epsilon_t,\epsilon_r\}$ is a basis of $E_F$.
        
        \item[(ii)] There is a choice of root of unity $\zeta \in \mu_F$ such that $u_t=\sqrt{\zeta\epsilon_t}\in O_F^\times$. Moreover, 
        the set $\{u_t,\epsilon_r\}$ is a basis of $E_{F}$ and  $\Log_F(u_t)$ is a generator of $\Lambda^t$.
       
         \item[(iii)] There is a choice of root of unity $\zeta \in \mu_F$ so that $u=\sqrt{\zeta\epsilon_t\epsilon_r}\in O_F^\times$. Moreover, the set $\{ u, u^\gamma \}$ is a basis is a basis of $E_F$, so $u$ is a Minkowski unit of $F$, and $\Log_F(\epsilon_t)$ is a generator of $\Lambda^t$. 
     \end{enumerate}
\end{corollary}
\begin{proof}
    Let $\Lambda$ and $\Lambda'$ be as in \Cref{prop: unit structure}. View the log unit lattice $\Lambda_{K}$ of $K$ as a sublattice of $\Lambda$, so $\Lambda_{K}\subseteq \Lambda^t$. Then by the definition of $\N_{F/K}$, we have $2\Lambda^t\subseteq \Lambda_{K}$. Hence, if $[\Lambda:\Lambda']=1$, we obtain the first two cases of \Cref{cor:unit-basis}.  If $[\Lambda:\Lambda']=2$, then $v$ as given in \eqref{eq:vectors} belongs to $\Lambda$ by \Cref{prop: unit structure}. By \Cref{remark: definition of vr and vt}, we have $\N_{F/K}\,v = v_t$, so $v_t \in \Lambda_{K}$.
    We conclude that $\Log_F(\epsilon_t)$ generates $\Lambda^t$, and one easily obtains (iii).  
\end{proof}

\begin{corollary}\label{cor: unit shapes in the boundary}
    Let $G\in \mathcal{G}$, with $\mathcal{G}$ as given in \eqref{eq:Galoisgroups}. Then $\Omega(D_4,2,1)\cup\Omega_{{\non-CM}}(G,0,3)$ is contained in the boundary of $\mathcal{S}_2$. Moreover, there are finitely many fields with a given index $2$ subfield as described in part 2 of \Cref{theorem: classification or rank 2} whose unit shape lies on the lower arc.
\end{corollary}
\begin{proof}
    For proving that the unit shapes from the statement are contained in the boundary, one can proceed exactly as in the proof of \cite[Proposition 3.1]{RNT}, where the authors stratified different regions of the fundamental domain w.r.t.\ the automorphism group of the lattice, and used the fact that the lattice admits a non-trivial automorphism.  Alternatively, one can directly compute the point $z\in \mathcal{S}_2$ corresponding to each of the cases in \Cref{cor: reduced basis} using the Gram matrix of each basis.  
    
    The finiteness of the family of $D_4$ number fields with a fixed real quadratic subfield and a well-rounded log unit lattice was proved in \cite[Corollary 1.8]{RNT}.     
   Now let consider a field $F \in \mathcal{F}_{{\non-CM}}(G,0,3)$, where $G \in \mathcal{G}$, with a fixed complex cubic subfield $F_3$. Let $\epsilon_t$ be the fundamental unit of $F_3$. Suppose that $\Lambda_F$ is well-rounded and let $\{v_1, v_2\}$ be a basis of $\Lambda_F$ consisting of shortest vectors. Then $\|v_1\|=\|v_2\|$.  The covolume of $\Lambda_F$ is $\sqrt{3} R_F$ where $R_F$ is the regulator of $F$. Hence
$$\sqrt{3} R_F \le \|v_1\| \|v_2\| =\|v_1\|^2 \le \|\Log_F(\epsilon_t)\|^2.$$
Since the cubic field $F_3$ is fixed, it follows that
the regulator of $F$ is bounded. Since the number of non-CM fields with bounded regulator is finite (see \cite[Theorem 1.1]{pazuki2014heights}), the result follows.
\end{proof}

\subsubsection{Transcendentality}\label{subsection: transcendentality}
As before, let $F/K$ be as in \Cref{prop: boundarytype}, and let $H=\Gal(F/K) = \langle \gamma \rangle$. Let $E\coloneq\hom(F,\C)$. We fix an ordering of the set $E$ and explicitly specify the embeddings on the ordered set $E$ defining the map  $\Log_F:O_F^\times\rightarrow \R^{3}$ used in the remainder of Section~\ref{sec:rank2}. Consider the equivalence relation $\sim$ on $E$ given by: $\rho\sim\rho'$ \text{if and only if} $\rho=\overline{\rho'}$, for $\rho,\rho'\in E.$ Since $\rk(\Lambda_F)=2$, we must have $\#(E/\sim)=3$.

The group $H$ acts on the set $E/\!\!\sim$\, via $\gamma\cdot[\rho]=[\rho\circ \gamma^{-1}]$. Since $H$ acts non-trivially on~$\Lambda_F$, it is easy to deduce from the Orbit-Stabilizer Theorem (see \cite[Prop.\ 2]{DummitFoote2004} for example) that this action has one orbit with only one equivalence class $\mathcal O_1\coloneq\{[\rho]\}$ and one orbit with two equivalence classes $\mathcal O_2\coloneq\{[\rho'],[\rho'']\}$. 
Fix a representative $\rho_1$ of the unique equivalence class $[\rho]$ in $\mathcal O_1$ and representatives $\rho_2$ and $\rho_3$ of the two equivalence classes $[\rho']$, $[\rho'']$ contained in $\mathcal O_2$. We explicitly define the map $\Log_F:O_F^\times\rightarrow \R^3$ as follows.
    \begin{equation} \label{eq:choice of log map}
    \Log_F(u)=\begin{bmatrix}
        \delta_1\log|\rho_1(u)|\\
        \delta_2\log|\rho_2(u)|\\
        \delta_3\log|\rho_3(u)|
    \end{bmatrix},\quad \delta_i=\quad\begin{cases}
        1&\text{if $\rho_i$ is real},\\
        2&\text{if $\rho_i$ is complex},
    \end{cases}
    \quad i=1,2,3.
    \end{equation}
    As before, the above choices are non-canonical, but two different choices for $\Log_F$ will produce isometric log unit lattices.

The following lemma is based on \cite[Step 1]{RNT}. In order for our paper to be self-contained, we provide our own exposition of the proof, although we follow the same procedure as in \textit{loc.cit}.

\begin{lemma} \label{lem:embedding}
    Let $F/K$ be as in \Cref{prop: boundarytype}, and let $u_1,u_2\in O_F^\times$ be such that $\Log_F(u_1)\in \Lambda_{F/K}^t$ and $\Log_F(u_2)\in \Lambda_{F/K}^r$. With the map $\Log_F$ given in \eqref{eq:choice of log map}, there exist $a,b\in \R$ such that
    \[
      \Log_F(u_1)=\begin{bmatrix}
          2a\\
          -a\\
          -a
      \end{bmatrix}\quad\text{and}\quad \Log_F(u_2)=\begin{bmatrix}
          0\\
          b\\
          -b
      \end{bmatrix}. 
    \]
\end{lemma}
\begin{proof}
 Recall that
\begin{equation}\label{eq: action}
    [\rho_2\circ \gamma]=[\rho_3]\quad \text{and}\quad [\rho_1\circ\gamma]=[\rho_1].
\end{equation}
Since $\Log_F(u_1)\in \Lambda_{F/K}^t$, we must have $\log|\rho_2(u_1)|=\log|\rho_3(u_1)|$. From here, it is easy to derive the claimed form for $\Log_F(u_1)$. To conclude the result for $u_2$, use \eqref{eq: action} and $u_2^{1+\gamma}\in \ker(\Log_F)$ to obtain
\[\begin{bmatrix}
    0\\
    0\\
    0
\end{bmatrix}=\Log_F(u_2^{1+\gamma})=\begin{bmatrix}
    2\delta_1\log|\rho_1(u_2)|\\
    \delta_2(\log|\rho_2(u_2)|+\log|\rho_3(u_2)|)\\
     \delta_3(\log|\rho_3(u_2)|+\log|\rho_2(u_2)|)\\
\end{bmatrix}.\]
It is now straightforward to derive the seerted form for $\Log_F(u_r)$.
\end{proof}

We now present two technical results that enable us to use the strategy in \cite[Proposition 3.4]{RNT}  for proving the second case of \Cref{theorem: classification or rank 2}. 
\begin{lemma}\label{lem: relative in the Galois closure}
    Suppose $F/K$ is as in \Cref{prop: boundarytype} and let $\Tilde{F}$ and $\Tilde{K}$ be the Galois closures of $F/\Q$ and $K/\Q$, respectively. Let $u_2\in O_F^\times$ be such that $\Log_F(u_2)\in \Lambda_{F/K}^r$ and $u_2\notin \ker(\Log_F)$. Then $\N_{\Tilde{F}/\Tilde{K}}(u_2)$ is a root of unity in $\Tilde{K}$. 
\end{lemma}
\begin{proof}
    Since $[F:K] = 2$, the minimal polynomial $m(x)$ of $u_2$ over $K$ is quadratic. Its constant coefficient is $\N_{F/K}(u_2)$, which is a root of unity by \Cref{lem:vtvrortho}. It suffices to prove that $m(x)$ is also the minimal polynomial of $u_2 \in \tilde{F}$ over $\tilde{K}$.
    To that end, note that the extension $\Tilde{K}/K$ has degree 1 if $K$ is quadratic and 2 if $K$ is cubic. It follows that $F \not\subseteq \tilde{K}$. In particular $u_2 \notin \tilde{K}$, so $m(x)$ does not factor over $\tilde{K}$.
\end{proof}

\begin{proposition}\label{prop: Qbarindependence}
Suppose $F/K$ is as in \Cref{prop: boundarytype}. Let $u_1,u_2\in O_F^\times$, be such that $u_1,u_2\notin\ker(\Log_F)$, $\Log_F(u_1)\in \Lambda_{F/K}^t$, and $\Log_F(u_2)\in \Lambda_{F/K}^r$. Choose an embedding $\rho: F \to \mathbb{C}$ such that $\log|\rho (u_2)| \neq 0$. Then $\log|\rho(u_1)|$ and $\log|\rho(u_2)|$ are $\bar{\mathbb{Q}}$-linearly independent.
\end{proposition}
\begin{proof}
   By Baker's Theorem on the linear independence of logarithms (\cite[Corollary 1]{BakerLinearIndependence}), it suffices to prove that $\log|\rho(u_1)|$ and $\log|\rho(u_2)|$ are $\Q$-linearly independent. 

Fix an extension of the embedding $\rho:F\rightarrow \C$ (again denoted $\rho$) to the Galois closure~$\Tilde{F}$ of $F/\Q$. Choose $\tau \in \Gal(\Tilde{F}/\mathbb{Q})$ such that $\tau$ acts as  complex conjugation on $\rho(\tilde{F})$; that is, $\rho(\tau(x))=\overline{\rho(x)}$ for every $x\in \Tilde{F}$.    Note that our choice of $\tau$ gives $|\rho(u_1)|^2=\rho(u_1^{1+\tau})$ and $|\rho(u_2)|^2=\rho({u_2}^{1+\tau})$. 

We claim that the $\Q$-linear independence of the set $\{\log|\rho(u_1)|,\log|\rho (u_2)|\}$ follows from the  assertion that the vectors $\Log_{\Tilde{F}}(u_1^{1+\tau})$ and $\Log_{\Tilde{F}}(u_2^{1+\tau})$ are non-zero and orthogonal. To see this, note that the map $\rho|_{K^\times}:K^\times\rightarrow \C^\times$ is injective, so the embedded units $\rho(u_1^{1+\tau})$ and $\rho(u_2^{1+\tau})$ are positive real numbers that are multiplicatively independent. Since $\log:\R_{>0}\rightarrow \R$ is injective, it takes multiplicatively independent sets to linearly independent sets in the abelian group $\R^\times$, proving our claim.

To prove the assertion, note that $\log|\rho(u_2)|\neq0$ by assumption, and $\log|\rho(u_1)|\neq 0$ as  otherwise \Cref{lem:embedding} would imply that $u_1\in \ker(\Log_{\Tilde{F}})$. To prove orthogonality, let $\Tilde{K}$ be the Galois closure of $K/\Q$. 
By \Cref{lem: relative in the Galois closure}, $\N_{\tilde{F}/\tilde{K}} (u_2)$ is a root of unity, so 
$\N_{\Tilde{F}/\tilde{K}}(u_r^{\tau})$ is also a root of unity, Hence, their product
\[ \N_{\tilde{F}/\tilde{K}}(u_2) \N_{\tilde{F}/\tilde{K}}(u_2^{\tau}) = \N_{\tilde{F}/\tilde{K}}(u_2^{1+\tau}) \]
is a root of unity in $\tilde{K}$. By \Cref{lem:vtvrortho}, $\Log_{\Tilde{F}}(u_1^{(1+\tau)})$ and $\Log_{\Tilde{F}}(u_2^{(1+\tau)})$ are orthogonal, which proves the assertion.  
\end{proof}
We are now ready to provide a proof for part~2 of \Cref{theorem: classification or rank 2}. We already established that for $F/K$ as in \Cref{prop: boundarytype}, the shape $[\Lambda_F]$ lies on the boundary of $\mathcal{S}_2$. It remains to show that $[\Lambda_F]$ is transcendental.

\begin{proposition}\label{theorem: trascendental point}
    
    Let \( F/K \) be as in \Cref{prop: boundarytype}. Then \( [\Lambda_F] \) is transcendental. 
\end{proposition}
\begin{proof} 
    Let $z = x+iy$ be the complex number in $\mathcal{S}_2$ representing $[\Lambda_F]$. 
    For the sake of contradiction, suppose $z$ is an algebraic number. Then the length of any vector in the lattice generated by the basis $\left \{\begin{bmatrix}
    1
    \\0
    \end{bmatrix}, 
    \begin{bmatrix}
        x \\ y 
    \end{bmatrix}\right \}$ 
    is also algebraic. Therefore, the ratio of the lengths of any two such vectors is algebraic. 
    By \Cref{prop: Qbarindependence}, for any two units $u_1,u_2 \in O_F^{\times}$, the ratio $\frac{\|\Log_F(u_1)\|}{\|\Log_F(u_2)\|}$  is a transcendental number. Since rotations, reflections, and scalings of lattices preserve the transcendentality of the ratio of the lengths of such vectors, this yields a contradiction. Consequently, 
    $z$ is transcendental.     
\end{proof}

We conclude this subsection with a necessary and sufficient condition for a transcendental unit shape on the boundary of $\mathcal{S}_2$ to be orthogonal.

\begin{corollary}\label{cor: characterization of orthogonality}
    Let $F/K$ be as in \Cref{prop: boundarytype}, and let $v_t$ and $v_r$ be generators of $\Lambda_{F/K}^t$ and $\Lambda_{F/K}^r$, respectively. Then $\Lambda_F$ is orthogonal if and only if $\sqrt{u_t u_r}\notin F$ for any choice
    of pre-images $u_t\in \Log_{F}^{-1}(\{v_t\})$ and $u_r\in \Log_{F}^{-1}(\{v_r\})$. Equivalently, $\Lambda_F$ is orthogonal if and only if $F$ does not admit a Minkowski unit. Moreover, if $\Lambda_F$ is well-rounded, then $F$ has a Minkowski unit.
\end{corollary}
\begin{proof}
   Suppose $\sqrt{u_tu_r}\notin F$ for any choice of pre-images  $u_t\in \Log_{F}^{-1}(\{v_t\})$ and $u_r\in \Log_{F}^{-1}(\{v_r\})$. Then  \Cref{cor:unit-basis}, followed by \Cref{cor: reduced basis}, implies that $\Lambda_F$ is orthogonal. 
   
   Conversely, assume that $v=(v_t+v_r)/2\in \Lambda_F$. Then \Cref{cor: reduced basis} shows that one of $\mathcal{B}_1\coloneq\{v_t,v\}$, $\mathcal{B}_2\coloneq\{v_r,v\}$, or $\mathcal{B}_3\coloneq\{v,\gamma v\}$ is a basis achieving the successive minima of~$\Lambda_F$. Assume by contradiction that $\Lambda_F$ is orthogonal. Then its basis consisting of successive minima must also be orthogonal. By computing the corresponding inner products, it is easy to see that neither~$\mathcal{B}_1$ nor~$\mathcal{B}_2$ are orthogonal; hence $\mathcal{B}_3$ must achieve the successive minima of~$\Lambda_F$. Since
   $\mathcal{B}_3$ is an orthogonal basis of $\Lambda_F$ consisting of vectors of the same length, it is easy to verify that $[\Lambda_F]=1\in \mathcal{S}_2$, contradicting \Cref{theorem: trascendental point}. 
   
   It remains to prove that if $\Lambda_F$ is well-rounded, then $F$ has a Minkowski unit. Note that the only well-rounded orthogonal lattice corresponds to the point $0+i\in \mathcal{S}_2$, which is the unique point in the intersection of the left boundary and the lower arc. Since $i$ is an algebraic number, we have $[\Lambda_F]\neq i$ by \Cref{theorem: trascendental point}. It follows that if $\Lambda_F$ is well-rounded, then $\Lambda_F$ is not orthogonal, and therefore $F$ has a Minkowski unit.
\end{proof}

\subsection{Interior rank 2 unit shapes}
We conclude this section with the proof of part 3 of \Cref{theorem: classification or rank 2}. This proof is contingent on the algebraic independence of logarithms, also known as the \emph{weak Schanuel Conjecture}.

\begin{conjecture}[Weak Schanuel Conjecture; \cite{Waldschmidt2023}, Conjecture 39.3.1] \label{conj: algebraic independence of logs}
    Let $\lambda_1,\dots,\lambda_m$ be $\Q$-linearly independent complex numbers such that the complex numbers $\alpha_i=e^{\lambda_i}$ are algebraic numbers. Then $\lambda_1,\dots,\lambda_m$ are algebraically independent. 
\end{conjecture}

We also require the notion of generic position from algebraic geometry.
\begin{definition}
     Let $\mathrm{Sym}_2(\C)$ be the algebraic variety of symmetric $2\times 2$ matrices over~$\C$. A point $A\in \mathrm{Sym}_2(\C)$ is said to be in \emph{$\overline{\Q}$-generic position} if $A$ is not contained in a closed proper subvariety of $\mathrm{Sym}_2(\C)$ defined over $\overline{\Q}$.
\end{definition}

We now have all the ingredients to prove part 3 of \Cref{theorem: classification or rank 2}. 

\begin{proposition}\label{prop: shape_space_totally_real}
 Assume Conjecture \ref{conj: algebraic independence of logs}. Let $F_3$ be a non-Galois totally real cubic field and let $F$ be a totally imaginary sextic field containing $F_3$. Let $A_{F_3}$ and $A_F$ denote arbitrary Gram matrices of $\Lambda_{F_3}$ and $\Lambda_F$, respectively. Then $A_{F_3}$ and $A_F$ are in $\overline{\Q}$-generic position in $\mathrm{Sym}_2(\C)$. It follows that for the complex numbers $z=x+iy$ and $z_{F_3}=x_{F_3}+iy_{F_3}$ representing $[\Lambda_F]$ and $[\Lambda_{F_3}]$ in $\mathcal{S}_2$, respectively, no proper closed subvariety of $\C^2$ defined over $\overline{\Q}$ contains the point $(x,y)$ or $(x_{F_3}, y_{F_3})$, which implies \textit{part 3} of \Cref{theorem: classification or rank 2}. 
\end{proposition}
\begin{proof}

We identify~$\Lambda_{F_3}$ with $\Log_{F}(O_{F_3}^{\times})$ via \Cref{prop: CMsublattice vs lattice}. First, note that for any Gram matrix of $\Lambda_{F_3}$, there exists a rational change of basis that produces a Gram matrix for~$\Lambda_F$. Therefore, $A_F$ is in $\overline{\Q}$-generic position if and only if $A_{F_3}$ is also in $\overline{\Q}$-generic position.

To prove that $A_{F_3}$ is in $\overline{\Q}$-generic position, let $\mathbb{H}^m\subseteq \mathbb{A}^{m+1}$ be the hyperplane defined by the equation $x_1+\dots+x_{m+1}=0$, and identify $\mathrm{Sym}_2(\C)$ with $\mathbb{A}^3$. Let $\varphi:\mathbb{H}^5\rightarrow \mathbb A^3$ be the following composition of maps:
\[\begin{tikzcd}
	\varphi:{\mathbb{H}^5} & {\mathbb{H}^2\times \mathbb{H}^2} & {\mathbb A^3} \\
	\begin{array}{c} \begin{bmatrix}x_1\\x_2\\x_3\\x_4\\x_5\\-\sum_{i=1}^5x_i\end{bmatrix}  \end{array} & (v_1, v_2)  & \begin{bmatrix}
    \langle v_1,v_1\rangle & \langle v_1,v_2\rangle\\
    \langle v_1,v_2\rangle & \langle v_2,v_2\rangle 
\end{bmatrix} ,
	\arrow[maps to, from=1-1, to=1-2]
	\arrow["{[\langle\cdot,\cdot\rangle]_{ij}}", from=1-2, to=1-3]
	\arrow[maps to, from=2-1, to=2-2]
	\arrow[maps to, from=2-2, to=2-3]
\end{tikzcd}\]
where 
\[ v_1 = \begin{bmatrix}x_1+x_2\\x_3+x_5\\-(x_1+x_2+x_3+x_5)\end{bmatrix}, \quad v_2 = \begin{bmatrix}-(x_1+x_2+x_4+x_5)\\ x_1+x_4\\x_2+x_5\end{bmatrix} .  \]
One can prove, via manual symbolic computation or using SageMath \cite{sagemath}, that the image of the map $\varphi$ is Zariski-dense in $\mathbb{A}^3$ by verifying that the ideal of $\C[x_1,x_2,x_3]$ vanishing on the image of~$\varphi$ is the zero ideal $\{0\}$.

The log unit lattice $\Lambda_{\Tilde{F_3}}$ of the Galois closure $\Tilde{F_3}$ of $F_3$ is a rank $5$ lattice that is also a $\Gal(\Tilde{F_3}/\Q)$-module. Choose an embedding $\rho:\Tilde{F_3} \rightarrow \R$, and identify 
\[ \Gal(\tilde{F_3}/\Q)\cong S_3 = \{ \text{id}, (12), (13), (23), (123), (132) \}. \]
By \Cref{thm: existence of weak mink unit}, there exists a unit $u\in O_{\Tilde{F_3}}^\times$ such that the set $\mathcal{B}\coloneq\{\log|\rho(g(u))|\}_{g\in S_3\setminus \{(132)\}}$ is linearly independent, and therefore algebraically independent by \Cref{conj: algebraic independence of logs}. Let $v=\Log_{\Tilde{F_3}}(u)$. The algebraic independence of $\mathcal{B}$ over $\Q$ implies that $v$ is not properly contained in any closed subvariety of $\mathbb{H}^5$ defined over $\overline{\Q}$. Now choose an ordering of the embeddings of $\tilde{F}_3$ into $\R$ 
such that
 \[v\coloneq\left[\log|\rho(u)|,\,\log|\rho(u^{(12)})|,\,\log|\rho(u^{(13)})|,\,\log|\rho(u^{(23)})|,\,\log|\rho(u^{(123)})|,\,\log|\rho(u^{(132)})|\right]^T.\]
It follows that the Gram matrix $M_{F_3}$ associated to the $\Z$-linearly independent set  
\[\{\Log_F(u^{1+(12)}),\Log_F(u^{(1+(12))(13)})\}\subseteq \Lambda_{F_3}\]  
satisfies $M_{F_3}=\varphi(v)$. We now assert that $M_{F_3}\notin Z$ for any Zariski closed subset $Z\subseteq \mathbb A^3$ defined over $\overline{\Q}$. To prove this assertion, note that if such a set $Z$ existed, then $\varphi^{-1}(M_{F_3})$ would be contained in $\varphi^{-1}(Z)$, which is impossible since $v$ does not belong to any Zariski-closed proper subset of $\mathbb{H}^5$ defined over $\overline{\Q}$. Since $A_{F_3}$ is related to $M_{F_3}$ by a rational change of basis, we conclude that any Gram matrix of $\Lambda_{F_3}$ is in $\overline{\Q}$-generic position in $\mathrm{Sym}_2(\C)$. 

Next, we prove that the point $(x,y)\in \C^2$ such that $z=x+iy$ represents the unit shape $[\Lambda_F]$ in $\mathcal{S}_2$ is not contained in any closed proper subvariety of $\C^2$ defined over $\overline{\Q}$. The proof of the analogous result for $[\Lambda_{F_3}]$ follows the same strategy, so we omit it.
 Choose a basis of $\Lambda_F$ such that the corresponding Gram matrix $A_F$  of $\Lambda_F$ has the form:
    \[
    \alpha\begin{bmatrix}
        1&x\\
        x&x^2+y^2
    \end{bmatrix}, \quad x+iy\in \mathcal S_2,\,\alpha\in \R_{>0}.
    \]
    Since $A_F$ is in $\overline{\Q}$-generic position in $\mathrm{Sym}_2(\C)$, the complex numbers $\alpha$, $x$, and $y$ must be algebraically independent over $\Q$. From this observation, we obtain that the point $(x,y)$ does not lie in any closed proper subvariety defined over $\overline{\Q}$ of $\C^2$. This concludes the proof.
\end{proof}


\section{Shapes of totally imaginary $D_6$ sextic number fields} \label{sec:shapesD6}
In this section we specialize to $\Omega(D_6, 0, 3)$, the set of unit shapes of totally imaginary $D_6$-sextic number fields, which by \Cref{theorem: classification or rank 2} splits into the disjoint CM and non-CM subfamilies $\Omega_{\mathrm{CM}}(D_6, 0, 3)$ and $\Omega_{\mathrm{non\text{-}CM}}(D_6,0,3)$, lying in the interior and boundary of $\mathcal{S}_2$ respectively. We address two natural questions about this family.

The first concerns the strength of the unit shape as an invariant: to what extent does the unit shape distinguish isomorphism classes within $\mathcal{F}(D_6,0,3)$? The CM and non-CM subfamilies exhibit strikingly different behavior. In the CM case, infinitely many non-isomorphic fields share the same unit shape, so the invariant carries little distinguishing power (\Cref{prop: real cubic shapes are CM shapes}). In the non-CM case, the unit shape turns out to be a complete invariant, separating isomorphism classes entirely (\Cref{theorem: Completeness of the shape}).

The second concerns the geometry of the shape space: what does $\Omega(D_6,0,3)$ look like as a subset of $\mathcal{S}_2$? Building on the location results of \Cref{theorem: classification or rank 2}, we go further and show that $\Omega(D_6,0,3)$ is non-discrete in $\mathcal{S}_2$ (see \Cref{theorem: limitpoint}), establishing in particular that each subfamily contains infinitely many distinct shapes. This is a first step toward the more ambitious goals of understanding the density or equidistribution of unit shapes within $\mathcal{S}_2$.


\subsection{$D_6$-sextic CM fields}
This family of fields falls under the scenario covered in \Cref{section: CM-Fields}. The unique cubic subfield of any field $F \in \mathcal{F}(D_6,0,3)$ is real, i.e.\ has signature $(3,0)$ and hence unit rank 2. We saw in \Cref{prop: the shape is weak} there exist only finitely many CM-extensions $F/K$ such that $[\Lambda_K]\neq [\Lambda_F]$. As a result, the space of unit shapes of totally real cubic fields with Galois group $S_3$ embeds into that of $D_6$-number fields with signature $(0,3)$, allowing the transfer of structural information between the two families.

\begin{proposition}\label{prop: real cubic shapes are CM shapes}
   For any totally real $S_3$-cubic field $K$, there exist infinitely many CM $D_6$-sextic fields containing $K$. Moreover, for all 
   but finitely many $F\in\mathcal{F}_{\CM}(D_6,0,3)$ with $K\subseteq F$, we have $[\Lambda_F]=[\Lambda_K]$. In particular, the family $\Omega(S_3,3,0)$ is contained in $\Omega_{\CM}(D_6,0,3)$.
  \end{proposition}
\begin{proof}
   Let $K$ be a totally real $S_3$-field. Then there are infinitely many squarefree positive integers $d$ such that the CM field $F = K(\sqrt{-d})$ of degree 6 is not Galois over $\Q$. Each of these fields $F$  contains a quadratic subfield and a cubic subfield with Galois group $S_3$. By \cite[Remark, p.~331]{cohen1993course}, the Galois group of the Galois closure of \(F/\mathbb{Q}\) is isomorphic to \(D_6\).  We conclude that there are infinitely many CM $D_6$-sextic fields containing $K$. By \Cref{prop: the shape is weak}, only finitely many of these fields satisfy $[\Lambda_F] \neq [\Lambda_{K}]$. In particular, $[\Lambda_K] \in \Omega_{\CM}(D_6,0,3)$, and therefore $\Omega({S_3},0,3) \subset \Omega_{\CM}(D_6,0,3)$. 
   \end{proof}

\begin{remark}\label{remark: cubics have limit points}
By \Cref{prop: real cubic shapes are CM shapes} we can view $\Omega(S_3,0,3)$ as a subspace of $\Omega(D_6,0,3)$, both endowed with the topology induced from $\mathcal{S}_2$. 
\end{remark}


\subsection{$D_6$-sextic non-CM fields}

We now turn to $\mathcal{F}_{\non-CM}(D_6,0,3)$, the family of non-CM $D_6$-sextic number fields. The unique cubic subfield of any field in this family is complex, i.e.\ has signature $(1,1)$) and hence unit rank 1. 
In \Cref{theorem: Completeness of the shape}, we establish that in contrast to the CM setting, fields in this family are classified, up to isomorphism, by their unit shape. 
As an auxiliary step, in \Cref{prop: DecompositionOfUnitGroup}, we decompose the $D_6$ representation $\Q\otimes_\Z\Lambda_{\Tilde{F}}$, where $\Tilde{F}$ is the Galois closure of some $F\in \mathcal{F}_{\non-CM}(D_6,0,3)$, into irreducible subrepresentations. Finally, in \Cref{theorem: limitpoint}, we exhibit a sequence in $\Omega_{\non-CM}(D_6,0,3)$ that converges to the hexagonal lattice in $\mathcal{S}_2$, thereby proving that the set $\Omega_{\non-CM}(D_6,0,3)$ is non-discrete. A remarkable property of this sequence
is that even though it converges to a well-rounded lattice, none of the members of the sequence are themselves well-rounded.

Throughout, we use the presentation
    \[  D_6 = \langle \sigma, \tau : \sigma^6 = \tau^2 = (\sigma\tau)^2 = 1 \rangle   \]
along with the following notation, illustrated in \Cref{fig: fielddiagram}. Let $F$ be a totally imaginary $D_6$-sextic field with Galois closure $\tilde{F}$. Let~$F_3$ and~$F_2$ be the complex cubic and imaginary quadratic subfields of $F$, respectively. We let $\Tilde{F_3}$ be the Galois closure of $F_3$, and $R_2$ the quadratic resolvent field of $F_3$ which lies inside $\Tilde{F}_3$ and is imaginary. Note that $\Gal(\Tilde{F_3}/\Q)\cong S_3$. We denote by $F_4$ the unique quartic subfield of $\Tilde{F}$; it contains $F_2$, $R_2$ and a third real quadratic subfield $K_2$ that is generated by the product of the generators of $F_2$ and $R_2$.

The group $D_6$ has three conjugacy classes of elements of order 2, namely $[\sigma^3], [\tau]$ and $[\sigma\tau]$, which correspond to three conjugacy classes of sextic fields in $\Tilde{F}$. It is clear that $\Tilde{F_3}=\mathrm{Fix}\langle\sigma^3\rangle$, the fixed field of $\langle\sigma^3\rangle$. The other two conjugacy classes are non-Galois sextics, and we may assume without loss of generality that $F=\mathrm{Fix}\langle\sigma\tau\rangle$. We also set $K_6\coloneq\mathrm{Fix}\langle\tau\rangle$. With these choices, $K_6$ contains a cubic subfield $K_3$ that is conjugate but not equal to $F_3$. Note that $K_6$ also contains the quadratic field $K_2$. 

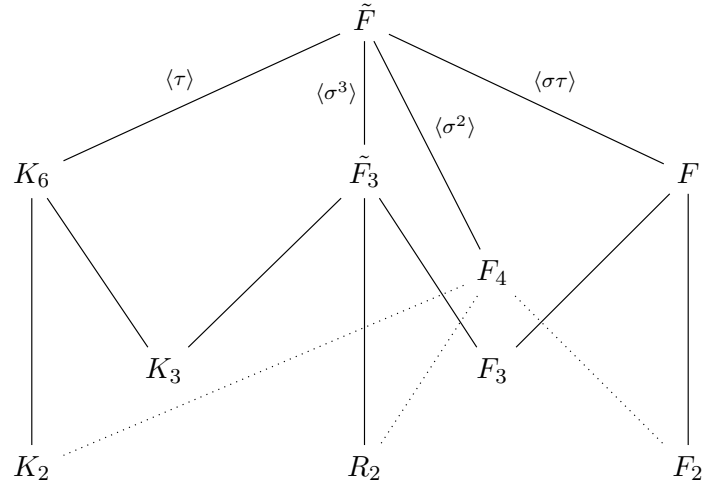
\begin{figure}[ht]
    \centering
\[\begin{tikzcd}
	&&& {\Tilde{F}} \\
	\\
	{K_6} &&& {\Tilde{F_3}} &&& F \\
	&&&& {F_4} \\
	& {K_3} &&& {F_3} \\
	{K_2} &&& {R_2} &&& {F_2}
	\arrow["{\langle\tau\rangle}"', no head, from=1-4, to=3-1]
	\arrow["{\langle \sigma^3\rangle}"', no head, from=1-4, to=3-4]
	\arrow["{\langle\sigma\tau\rangle}", no head, from=1-4, to=3-7]
	\arrow["{\langle \sigma^2\rangle}", no head, from=1-4, to=4-5]
	\arrow[no head, from=3-1, to=5-2]
	\arrow[no head, from=3-1, to=6-1]
	\arrow[no head, from=3-4, to=5-2]
	\arrow[no head, from=3-4, to=5-5]
	\arrow[no head, from=3-4, to=6-4]
	\arrow[no head, from=3-7, to=5-5]
	\arrow[no head, from=3-7, to=6-7]
	\arrow[dotted, no head, from=4-5, to=6-1]
	\arrow[dotted, no head, from=4-5, to=6-4]
	\arrow[dotted, no head, from=4-5, to=6-7]
\end{tikzcd}\]
    \caption{Subset of the subfield diagram for $\Tilde{F}$, the Galois closure of $F \in \mathcal{F}(D_6, 0, 3)$. }
    \label{fig: fielddiagram}
\end{figure}

Since $F_3$ is complex, we may choose an embedding $\rho:\Tilde{F}\hookrightarrow \C$ such that $\tau$ acts as complex conjugation on $\Tilde{F}$. It follows that for any $u\in \Tilde{F}$ we have $|\rho(u)|^2=\rho(u^{1+\tau})$. This matches our use of the notation $\tau$ in \Cref{subsection: transcendentality}.

\subsubsection{The representation $\Q\otimes_\Z\Lambda_{\Tilde{F}}$}

As in \Cref{remark: sublattice vs lattice}, for every extension of number fields $L/M$, we will always view $\Lambda_M$ as a sublattice of $\Lambda_L$ and write $\Lambda_M\subset\Lambda_L$, even though we may have $[\Lambda_M]\neq [\Log_L(O_M^\times)]$. For $F \in \mathcal{F}_{\non-CM}(D_6,0,3)$, we will also view the inner product space $\Q\otimes_\Z\Lambda_{\Tilde{F}}$ as a module over $\Q[D_6]$, or equivalently, as a representation of the group $D_6.$ It will be useful to note that $\Q\otimes_\Z \Lambda_{F_4}=\Q\otimes_\Z \Lambda_{K_2}$; this is true since both $\Lambda_{F_4}$ and $\Lambda_{K_2}$ are rank 1 lattices and $\Lambda_{K_2}\subseteq \Lambda_{F_4}$. Note that the extension $F/F_3$ is Galois; more specifically, we have
\begin{equation}\label{eq: fixedFields}
    F=\mathrm{Fix}\langle \sigma\tau\rangle,\quad \text{and} \quad \Gal(F/F_3)=\langle \tau\sigma^2|_F\rangle.
\end{equation}

Before decomposing the representation $\Q\otimes_\Z\Lambda_{\Tilde{F}}$, we must first understand the representation $\Q\otimes_\Z\Lambda_{\Tilde{F}_3}$.

\begin{lemma}\label{lem: the representation}
    Let $W_6\coloneq\Q\otimes_\Z\Lambda_{\Tilde{F_3}}$. The action of  $\Gal(\Tilde{F}/\Q)$ on $\Lambda_{\Tilde{F_3}}$ yields an irreducible orthogonal representation 
    $\Tilde{\mathbf{s}}:\Gal(\Tilde{F}/\Q)\rightarrow \mathrm{GL}(W_6)$ of degree $2$.
   Moreover, we have the following commutative diagram:
    \[\begin{tikzcd}
	{D_6\cong\Gal(\Tilde{F}/\Q)} \\
	{D_6/\langle \sigma^3\rangle\cong\Gal(\Tilde{F_3}/\Q)} & {\mathrm{GL}(W_6)}
	\arrow[from=1-1, to=2-1]
	\arrow["{{\Tilde{\mathbf{s}}}}", from=1-1, to=2-2]
	\arrow["{\mathbf{s}}"', dashed, from=2-1, to=2-2]
\end{tikzcd}\]
where the representation $\mathbf{s}:D_6/\langle \sigma^3\rangle\rightarrow \mathrm{GL}(W_6)$ is isomorphic to the unique irreducible two-dimensional representation of the symmetric group $S_3$, modulo isomorphism. 
\end{lemma}
\begin{proof}
    The representation $\Tilde{\mathbf{s}}$ is orthogonal since $\Gal(\Tilde{F_3}/\Q)$ acts on $\Lambda_{\Tilde{F_3}}$ by permuting the coordinates. Given that $\Tilde{F_3}$ is the fixed field by $\langle\sigma^3\rangle$, it is clear that $\tilde{\mathbf{s}}$ factors through the quotient map $D_6\rightarrow D_6/\langle\sigma^3\rangle$. To complete the proof of the lemma, it suffices to show that the corresponding representation 
    \[
    \mathbf{s}:D_6/\langle \sigma^3\rangle\cong \Gal(\Tilde{F_3}/\Q)\rightarrow \mathrm{GL}(W_6)
    \]
    is irreducible, since $S_3$ only has one irreducible representation of degree $2$  modulo isomorphism.  To this end, we prove that the restriction of the map $\mathbf{s}$ to the subgroup $\langle\sigma\rangle/\langle \sigma^3\rangle\subset S_3$ is irreducible, which will yield the result. 
    
    By \Cref{thm: existence of weak mink unit} we can choose a unit $u\in O_{\Tilde{F_3}}^\times$ such that $u$ and its conjugates generate a subgroup of $E_{\Tilde{F_3}}$ of finite index.  It follows that $W_6=\mathrm{span}_{\Q[S_3]}(\Log_{\Tilde{F_3}}(u))$. The element $1+\sigma+\sigma^2$ annihilates~$\Lambda_{\Tilde{F_3}}$. To see this, note that $u^{1+\sigma+\sigma^2}$ lies in the imaginary quadratic subfield $R_2$ of $\Tilde{F_3}$, as $u^{1+\sigma+\sigma^2}$ is fixed by the automorphism~$\sigma$ that generates the unique subgroup of order 3 in $S_3$. It follows that $u^{1+\sigma+\sigma^2}$ a root of unity. We conclude that $W_6$ is a vector space
    over the field $\Q[\sigma]/(1+\sigma+\sigma^2)$. Since $W_6$ and the field $\Q[\sigma]/(1+\sigma+\sigma^2)$ are both two-dimensional $\Q$-vector spaces, they must be isomorphic as $\Q[\sigma]/(1+\sigma+\sigma^2)$-vector spaces; hence $\Q[\sigma]/(1+\sigma+\sigma^2)\cong W_6$ as $\left(\langle \sigma\rangle/\langle\sigma^3\rangle\right)$-representations.
    Because $\Q[\sigma]/(1+\sigma+\sigma^2)$ is an irreducible representation of the group $\langle \sigma\rangle/\langle \sigma^3\rangle$, it is also irreducible as an $S_3$-representation. This concludes the proof.
\end{proof}

\begin{proposition}\label{prop: DecompositionOfUnitGroup}
    Let $V\coloneq\Q\otimes_\Z \Lambda_{\Tilde{F}}$, and let $V_{\rel}$ be the orthogonal complement of $\Q\otimes_\Z(\Lambda_{\Tilde{F_3}}\oplus \Lambda_{F_4})$ in $V$. Then
    \[
    V=(\Q\otimes_\Z\Lambda_{{\Tilde{F_3}}})\oplus (\Q\otimes_\Z\Lambda_{F_4})\oplus V_{\rel}
    \]
   as orthogonal $D_{6}$-representations. Moreover, $\Lambda_{F/F_3}^r \subseteq V_{\rel}$, with $\Lambda_{F/F_3}^r$ as in~\Cref{lem:vtvrortho}.
 \end{proposition}
\begin{proof}
    The action of $\Gal(\Tilde{F}/\Q)$ on $W_2\coloneq\Q\otimes_\Z \Lambda_{K_2}$ of $V$ yields a one-dimensional representation. We denote the corresponding character by $\chi_2$. Note also that $W_2=\Q\otimes_\Z \Lambda_{F_4}$.
     As before, let $W_6\coloneq\Q\otimes_\Z\Lambda_{\Tilde{F_3}}$. By Lemma \ref{lem: the representation}, the representation $\Tilde{\mathbf{s}}:\Gal(\tilde{F}/\Q)\rightarrow \mathrm{GL}(W_6)$ is irreducible and not isomorphic to the representation $\chi_2$. Since the action of $\Gal(\tilde{F}/\Q)$ on $V$ is orthogonal, we obtain that $\Q\otimes_\Z\Lambda_{\Tilde{F_3}}$ and $\Q\otimes_\Z \Lambda_{F_4}$ are orthogonal vector spaces, which yields the asserted direct sum decomposition of $V$.

    Now let $v_r \in \Lambda_{F/F_3}^r$; we need to prove that $v_r\in V_{\mathrm{rel}}$. Recall from \Cref{eq: fixedFields} that $F=\mathrm{Fix}\langle\sigma\tau\rangle$, so $\sigma\tau v_r=v_r$.
    Thus, it suffices to prove that $v_{r}$ is orthogonal to $\Lambda_{\Tilde{F_3}}$ and $\Lambda_{F_4}$. 
    
    To see that $v_r$ is orthogonal to $\Lambda_{\Tilde{F_3}}$, we compute the relative norm $\N_{\Tilde{F}/\Tilde{F_3}}$ applied to $v_r$ to obtain
    \[
    (1+\sigma^3)v_r=(1+\sigma^3\sigma\tau)v_r=(1+\tau\sigma^2)v_r=0.    
    \]
    The last equality follows from the identity $\sigma^4\tau = \tau \sigma^2$ and the fact that $\N_{F/F_3}(u)=u^{1+\tau\sigma^2}$ for all $u\in O_F^\times$. The orthogonality of $v_r$ and $\Lambda_{\Tilde{F_3}}$ now follows from  \Cref{lem: orthogonality of relative units} (with $\tilde{F}_3$ in place of $F'$).
    
    To see that $v_r$ is orthogonal to $\Lambda_{F_4}$, it is sufficient to prove that $v_r$ is orthogonal to $\Lambda_{K_2}$ because $\Q\otimes_\Z \Lambda_{K_2}=\Q\otimes_\Z\Lambda_{F_4}$. The Galois group of $\tilde{F}/K_2$ is $\Gal(\tilde{F}/K_2) = \langle \sigma^2, \tau \rangle \cong S_3$. Again using the identity  $\tau \sigma^2 = \sigma^4\tau$, applying the relative norm $\N_{\tilde{F}/K_2}$ to $v_r$ yields 
    \[
    \begin{aligned}
        (1+\tau+\sigma^2+\tau\sigma^2+\sigma^4+\tau\sigma^4)v_r&=(1+\tau\sigma^2)v_r+\tau(1+\tau\sigma^2)v_r+\sigma^4(1+\tau\sigma^2)v_r\\
        &=0+0+0 = 0.
    \end{aligned}
    \]
    Once again, \Cref{lem: orthogonality of relative units} (with $F' = K_2)$ implies that $v_r$ is orthogonal to $\Lambda_{K_2}$, which proves the result.
\end{proof}

\subsubsection{The unit shape as a classifying invariant}\label{subsection: shapecompleteinvariant}

Using \Cref{prop: DecompositionOfUnitGroup}, we can now prove that conjecturally, the unit shape of any field 
$F \in \mathcal{F}_{\text{non-CM}}(D_6, 0, 3)$ uniquely determines $F$ up to isomorphism. This assertion assumes a conjecture known as the \textit{Four Exponentials Problem}; a survey of this conjecture can be found in \cite{Waldschmidt2023}.

\begin{conjecture}[See the paragraph above Conjecture 39.4.1 of \cite{Waldschmidt2023}]\label{conj: four exponentials problem}
   Let  $\alpha_1,\alpha_2, \alpha_3$ and $\alpha_4$ be logarithms of algebraic numbers. Then 
   \[
   \det\left(\begin{bmatrix}
       \alpha_1&\alpha_2\\
       \alpha_3&\alpha_4
   \end{bmatrix}\right)=0
   \]
   if and only if either the two rows or the two columns are linearly dependent over $\Q$.

\end{conjecture}

\begin{theorem}\label{theorem: Completeness of the shape}
  Let $F, F' \in \mathcal{F}_{\text{non-CM}}(D_6, 0, 3)$. 
 Assuming \Cref{conj: four exponentials problem}, we have $F\cong F'$ if and only if $[\Lambda_F]=[\Lambda_{F'}]$.
\end{theorem}

\begin{proof}
 If $F\cong F'$, then it is clear that $[\Lambda_F]=[\Lambda_{F'}]$. So suppose now that $[\Lambda_F]=[\Lambda_{F'}]$. Let $F_3$ and $F_3'$ be the respective cubic subfields of $F$ and $F'$. Let $v_r=\Log_F(u_r)\in \Lambda_{F}$ and $v_r'=\Log_{F'}(u_r')\in \Lambda_{F'}$ correspond to generators of the respective subgroups of relative units of $F$ and $F'$, and choose $u_t\in O_F^\times$ and $u_t'\in O_{F'}^\times$ such that $v_t=\Log_F(u_t)$ and $v_t'=\Log_{F'}(u_t')$ are generators of $\Lambda_{F/F_3}^t$ and $\Lambda_{F'/F_3'}^t$, respectively, with $\Lambda_{F/F_3}^t$ and $\Lambda_{F/F_3}^r$ as defined in ~\Cref{lem:vtvrortho}.
 
 Let $\widetilde{FF'}$ be the Galois closure of the composite field of $F$ and $F'$, and choose an embedding $\tilde{\rho}:\widetilde{FF'}\rightarrow \C$.  To avoid cumbersome notation, we henceforth write $|u|$ instead of $|\tilde{\rho}(u)|$ for any $u\in \widetilde{FF'}$. By replacing $F$ and $F'$ by one of their conjugate fields if necessary, we may assume that all the numbers $\log|u_t|,\log|u_r|,\log|u_t'|$ and $\log|u_r'|$ are non-zero. An easy computation
 using \Cref{lem:embedding} shows that the Gram matrices corresponding to the linearly independent sets $\{\Log_{F}(u_t),\Log_{F}(u_r)\}$ and $\{\Log_{F'}(u_t'),\Log_{F'}(u_r')\}$ are given by 
  \[
 \begin{bmatrix}
            6\log^2|u_t|& 0 \\
            0 &2\log^2|u_r|
        \end{bmatrix}\quad\text{and}\quad \begin{bmatrix}
            6\log^2|u_t'|&0 \\
            0&2\log^2|u_r'|
        \end{bmatrix},
 \]
respectively. Since $[\Lambda_F]=[\Lambda_{F'}]$, there exists $\lambda\in \R_{>0}$ and $A=\begin{bmatrix}
            a&b\\
            c&d
        \end{bmatrix}\in \mathrm{GL}_2(\Q)$ such that
        \[
        A\begin{bmatrix}
            6\log^2|u_t|& 0 \\
            0 &2\log^2|u_r|
        \end{bmatrix}A^T=\lambda\cdot \begin{bmatrix}
            6\log^2|u_t'|& 0 \\ 
            0&2\log^2|u_r'|
        \end{bmatrix}.
        \]
       Explicitly performing the matrix multiplication on the left hand side, we obtain
\[ \begin{bmatrix}
        6a^2\log^2|u_t|+2b^2\log^2|u_r|&6ac\log^2|u_t|+2bd\log^2|u_r|\\
        6ac\log^2|u_t|+2bd\log^2|u_r|&6c^2\log^2|u_t|+2d^2\log^2|u_r|\\
    \end{bmatrix} = \lambda \cdot \begin{bmatrix}
            6\log^2|u_t'|& 0 \\ 
            0&2\log^2|u_r'|
        \end{bmatrix}. \]
By \Cref{prop: Qbarindependence}, the quantities $\log^2|u_t|$ 
and $\log^2|u_r|$ are $\bar{\Q}$-linearly independent, so $ac=bd=0$. Furthermore, since $A$ is invertible, either $c=b=0$ or $a=d=0$. So either
\begin{align*}
& 6a^2\log^2|u_t|=6\lambda\log^2|u_t'|\quad\text{and}\quad 2d^2\log^2|u_r|=2\lambda\log^2|u_r'|, \quad \text{or} \\
& 2b^2\log^2|u_r|=6\lambda\log^2|u_t'|\quad\text{and}\quad 6 c^2\log^2|u_t|=2\lambda\log^2|u_r'|.
\end{align*}
Solving for $\lambda$, one of the following equations must hold: 
    \begin{equation} \label{eqn: relation ur and ut}
        \begin{aligned}
            a^2(\log|u_r'|)^2(\log|u_t|)^2 &= d^2(\log|u_t'|)^2(\log|u_r|)^2\quad\text{or}\\ 
            b^2(\log|u_r|)^2(\log|u_r'|)^2 &= 9c^2(\log|u_t|)^2(\log|u_t'|)^2.
        \end{aligned}
    \end{equation}
    We prove that \eqref{eqn: relation ur and ut} implies $F\cong F'$. 

    Let $\Tilde{\tau}\in \Gal(\widetilde{FF'}/\Q)$ be the automorphism on $\widetilde{FF'}$ induced by complex conjugation. Note that 
    \begin{equation}\label{eq: absvalue}
    u^{1+\Tilde{\tau}}=    |u|^2 \quad \text{for all $u\in \widetilde{F F'}$}.
    \end{equation}
    Assume for a contradiction that $F\not\cong F'$. Then $M\coloneq\Tilde{F}\cap \Tilde{F'}$ is a Galois number field of degree $m\coloneq[M:\Q] \le 12$. The only possibilities for $m$ are $m=1,2,4,6$ or $12$. We show that Equations \eqref{eqn: relation ur and ut} cannot hold for any of these values of $m$. 

    First, if $m=1,2$ or $4$, then $M$ is contained  in the unique biquadratic field $F_4 \subset \Tilde{F}$. By Proposition \ref{prop: DecompositionOfUnitGroup}, the sublattices $\Lambda_{\rel}$ and $\Lambda_{\Tilde{F_3}}$ are orthogonal to $\Lambda_{F_4}$; in particular, their respective intersections with $\Lambda_{F_4}$ are trivial. It follows that 
    \[
    \Log_{\widetilde{FF'}}(\N_{\widetilde{F F'}/F_4}(u^{1+\tilde{\tau}}))=0 \quad \text{for all $u\in\{u_t,u_r,u_t',u_r'\}$}.
    \]
    Since both $\Tilde{F}$ and $\Tilde{F'}$ are Galois, we may use \Cref{lem: orthogonality of relative units} to conclude that the set 
    \[\mathcal{B}\coloneq\big\{(1+\tilde{\tau})\Log_{\widetilde{FF'}}(u) : u\in \{u_t,u_r,u_t',u_r'\}\big\}\]
    is an orthogonal set. Our choice for the embedding $\tilde{\rho}:\widetilde{FF'}\rightarrow \C$ guarantees that no element of $\mathcal{B}$ is the zero vector. Using Equation \eqref{eq: absvalue} and the fact that $\log:\R_{>0}\rightarrow \R$ is injective, we conclude that the subset $\{\log|u_t|,\log|u_r|,\log|u_t'|,\log|u_r'|\}\subset \R$ is $\Q$-linearly independent. \Cref{conj: four exponentials problem} now implies that \eqref{eqn: relation ur and ut} cannot hold. 

    Next, suppose that $m=6$. Since $M/\Q$ is a Galois extension, we must have $M=\Tilde{F_3}$. It follows that the cubic subfields of $\Tilde{F}$ and $\Tilde{F'}$ are isomorphic, so $\log|u_t|^2=\log|u_t'|^2$. Hence, \eqref{eqn: relation ur and ut} can be simplified to
    \[
    a^2(\log|u_r'|)^2=d^2(\log|u_r|)^2\quad\text{or}\quad b^2(\log|u_r|)^2(\log|u_r'|)^2=9c^2(\log|u_t|)^4.
    \]
    As in the case $m \le 4$, \Cref{prop: DecompositionOfUnitGroup} and Equation \eqref{eq: absvalue} imply that the set $\mathcal{B}'\coloneq\{\log|u_r'|,\log|u_r|,\log|u_t|\}$ is linearly independent. \Cref{conj: four exponentials problem} again shows that both equations in \eqref{eqn: relation ur and ut} are impossible. 

    Finally, if $m=12$, then $\Tilde{F}=\Tilde{F'}$. Note that $D_{6}$ has three conjugacy classes of subgroups of order 2. Two of these correspond to two distinct conjugacy classes of non-Galois sextic subfields of $\Tilde{F}$, namely the classes of $F$ and $K_6$. The quadratic subfield $K_2$ of $K_6$ and its conjugate fields are real, while the quadratic subfield $F_2$ of $F$ and its conjugates are imaginary. Since $F$ and $F'$ are non-Galois and both contain only imaginary quadratic subfields, they must be conjugate. This proves the result.     
\end{proof}


\subsubsection{A convergent sequence of non-CM $D_6$-unit shapes} 
We conclude \Cref{sec:shapesD6} by proving (in \Cref{theorem: limitpoint}) that the unit shapes of a family of imaginary pure sextic number fields investigated by Stender \cite{Stender1977} lie on the right boundary od $\mathcal{S}_2$ and converge in the standard fundamental domain $\mathcal{S}_2$ to the algebraic number $z=\frac{1}{2}+\frac{\sqrt{3}}{2}i$, which corresponds to the hexagonal lattice. It follows that the hexagonal lattice is contained in the closure of $\Omega_{\non-CM}(D_6,0,3)$ inside $\mathcal{S}_2$. Even though the unit groups in this family contain a Minkowski unit, we will show that their log unit lattices are never well-rounded. 

Let $\omega \in \C$ such that $\omega^6$ is a negative integer, so \(F=\mathbb{Q}(\omega)\in\mathcal{F}_{\text{non-CM}}(D_6,0,3)\). Let \(\tau_i\colon F\to\mathbb{C}\) for \(i\in\{1,2,3\}\) be representatives of the three pairs of complex conjugate embeddings of \(F\). Writing \(\zeta=e^{\pi i/3}\), we choose these embeddings so that \(\tau_1(\omega)=\omega\), \(\tau_2(\omega)=\omega\zeta\), and \(\tau_3(\omega)=\omega\zeta^2\). 
 The logarithmic embedding $\Log_F : F^\times \longrightarrow \mathbb{R}^3$ corresponding to these representatives is given by
\[
\Log_F(\alpha) = \begin{bmatrix} 2\log|\tau_1(\alpha)| \\ 2\log|\tau_2(\alpha)| \\ 2\log|\tau_3(\alpha)| \end{bmatrix}.
\]

We focus on a specific family of pure fields in $\mathcal{F}_{\text{non-CM}}(D_6, 0, 3)$, for which Stender \cite{Stender1977} provided explicit fundamental units.

\begin{proposition}[\mbox{\cite[Satz 19]{Stender1977}}] \label{prop:stender77}

Let $D$ be a positive integer such that $m = D^6 \pm 1 > 1$ is cubefree, and let $\omega = \sqrt[6]{-27m}$. Then 
\[ \epsilon_1 = D^2 - D\omega+ \omega^2/3, \qquad \epsilon_2 = D^2 + D\omega + \omega^2/3 \]
is a pair of fundamental units of 
$F = \mathbb{Q}(\omega)$.
\end{proposition}

\begin{theorem}\label{theorem: limitpoint}
 The unit shapes of the family of fields $F = \mathbb{Q}(\omega)$ of \Cref{prop:stender77} lie on the right boundary of $\mathcal{S}_2$ and converge to the hexagonal lattice shape as $D \to \infty$. Moreover, $\Lambda_F$ is not well-rounded for any $m$ of the form given in \Cref{prop:stender77}. 
\end{theorem}   
\begin{proof}
    Let $\{\epsilon_1, \epsilon_2\}$ be the pair of fundamental units of $F$ given in \Cref{prop:stender77}. Let $\omega = \sqrt[6]{-27m}$ with $m = D^6 + 1$ (the reasoning for $m = D^6 - 1$ is completely analogous). 
    Here, the cubic subfield of $F$ is \(F_3=\mathbb{Q}(\omega^2)\). Let \(\gamma\) denote the non-trivial automorphism of \(\Gal(F/F_3)\), so \(\gamma(\omega)=-\omega\). It follows that \(\epsilon_1\) is a Minkowski unit of \(F\). Fix~$i$ as the choice of 6-th root of $-1$ such that $\omega = i \sqrt{3} \sqrt[6]{D^6+1}$, and let $\zeta = e^{\pi i/3} = (1 + i\sqrt{3})/2$. Then
\begin{align*}
\Log_F(\epsilon_1)
&=
\begin{bmatrix}
2\log \left| D^2 - D\omega + \tfrac{\omega^2}{3} \right| \\[6pt]
2\log \left| D^2 - D\omega\zeta + \tfrac{\omega^2\zeta^2}{3} \right| \\[6pt]
2\log \left| D^2 - D\omega\zeta^2 - \tfrac{\omega^2\zeta}{3} \right|
\end{bmatrix}, \\
\Log_F(\epsilon_2)
&=
\Log_F(\gamma(\epsilon_1))
=
\begin{bmatrix}
2\log \left| D^2 + D\omega + \tfrac{\omega^2}{3} \right| \\[6pt]
2\log \left| D^2 + D\omega\zeta + \tfrac{\omega^2\zeta^2}{3} \right| \\[6pt]
2\log \left| D^2 + D\omega\zeta^2 - \tfrac{\omega^2\zeta}{3} \right|
\end{bmatrix}.
\end{align*}
Defining $b_1=\log | D^{2}-D\omega+\omega^{2}/3 |$ and $b_2=\log | D^{2}-D\omega\zeta+\omega^{2}\zeta^{2}/3|$, we have
\[
\Log_F(\epsilon_1)= 2\begin{bmatrix} b_1 \\ b_2 \\-b_1-b_2 \end{bmatrix}, 
\qquad
\Log_F(\epsilon_2)= 2\begin{bmatrix} b_1 \\ -b_1-b_2 \\ b_2\end{bmatrix}.
\]
Let $Y= \sqrt[6]{1 + D^{-6}} \in \R$. Using the identity $\omega = i\sqrt{3}DY$, we compute $b_1$ and $b_2$ as follows.
\begin{align*}
b_1&=\log | D^2(1 - Y^2 -i\sqrt{3}Y) | = 2\log D + \tfrac{1}{2} \log  (1 + Y^2 + Y^4),  \\[2pt]
b_2 &= 2 \log D + \log \left ( 1 + Y \right ) + \tfrac{1}{2} \log \left ( 1 + Y +Y^2 \right ).
\end{align*}
Let $v'=\Log_F(\epsilon_1)+\Log_F(\epsilon_2)$. Then  $\{v', \Log_F(\epsilon_1)\}$ is also a basis of $\Lambda_F$. Moreover,
\begin{align*}
\|v'\|^2 &= 24 b_1^2,  \\
\|\Log_F(\epsilon_1)\|^2 
  &= \|\Log_F(\epsilon_2)\|^2 
   = 8 \bigl (b_1^2 + b_1 b_2 + b_2^2\bigr).
\end{align*}
Since \( D \ge 1 \), we have \( Y > 1 \), and hence
\[
\exp(2b_1 - 2b_2)
= \frac{1 + Y^2 + Y^4}{(1+Y)^2(1+Y+Y^2)} < 1.
\]
It follows that \( b_1 < b_2 \). Consequently,
\[
\|v'\|^2 = 24 b_1^2 < \|\Log_F(\epsilon_1)\|^2 = \|\Log_F(\epsilon_2)\|^2.
\]
Therefore, the set \( \{\Log_F(\epsilon_1), \Log_F(\epsilon_2)\} \) does not attain the successive minima of \( \Lambda_F \). By \Cref{cor: reduced basis}, the lattice \( \Lambda_F \) is not well-rounded. Since $F$ has a Minkowski unit, it is not orthogonal by \Cref{cor: reduced basis} (i). This implies that the unit shape of $F$ is on the right boundary of $\mathcal{S}_2$.

Next, we compute the limit point of Stender's family. Let $\phi$ denote the angle between $v'$ and $\Log_F(\epsilon_1)$. Since $b_2/b_1>1$,   we obtain
\[ 
\cos(\phi) =  \frac{\sqrt{3}}{ 2  \sqrt{(b_2/b_1)^2+b_2/b_1+1}}\in (0, 1/2).\]
As \( D \to \infty \), the ratio \( b_2/b_1 \) approaches \(1\), which implies that $\cos(\phi)$ tends to \( 1/2 \). Hence \( \|\Log_F(\epsilon_1)\|^2/\|v'\|^2 \to 1 \). Consequently, the associated lattices converge to the hexagonal lattice. 
\end{proof}


\section{Orthogonality of some families of imaginary pure sextics} \label{sec:geomD6}

In this section, we investigate the orthogonality of the log unit lattices for a family of imaginary pure sextic fields $F = F_3F_2$ that are the compositum of a pure cubic field $F_3 = \mathbb{Q}(\sqrt[3]{m})$ and an imaginary quadratic field $F_2= \mathbb{Q}(\sqrt{-d})$ for certain positive integers~$d, m$. In \Cref{ss:index}, we explicitly determine the ring of integers $O_F$. In \Cref{sec:orth}, we construct a polynomial whose coefficients depend on $d$ and $m$ for which the absence of an integer root modulo a certain large factor of $d$ guarantees the orthogonality of the log unit lattice of $F$, implying that the shapes of these lattices lie on the left boundary of~$\mathcal{S}_2$. Finally, in \Cref{sec:density}, we also provide a lower bound on the proportion of orthogonal lattices in the subfamily for $m=2$ and establish the exact density of fields with $m = 2$ and $d$ prime.

\subsection{The ring of integers of $F = \Q(\sqrt[3]{m}, \sqrt{-d})$} \label{ss:index} 

Under certain conditions on $m$ and $d$, the ring of integers $O_F$ can be explicitly determined from the subrings of integers $O_{F_3}$ and $O_{F_2}$. We provide explicit generators of $O_F/O_{F_3}$ in these cases. 

\begin{lemma}\label{lem:index1}
    Let $d, m$ be integers with $d > 0$, $d$ squarefree, $3 \nmid d$, $m$ odd, $m$ cubefree and $\gcd(d,m) = 1$. Let $F_2 = \mathbb{Q}(\sqrt{-d})$, $F_3 = \mathbb{Q}(\sqrt[3]{m})$ and $F = F_3 F_2= \mathbb{Q}(\sqrt[3]{m}, \sqrt{-d})$. Then $O_F = O_{F_3} \otimes_\Z O_{F_2} = O_{F_3}[\beta]$  where
    \[ \beta = \begin{cases} -d+\sqrt{-d} & \mbox{if $d \equiv 1, 2 \pmod{4}$}, \\[3pt] \displaystyle \frac{-d+\sqrt{-d}}{2} & \mbox{if $d \equiv 3 \pmod{4}$}. \end{cases} \]
    Moreover, $\Tr_{F/F_3}(\beta)$ and $\N_{F/F_3}(\beta)$ are divisible by $d$ in $O_{F_3}$.
   \end{lemma}
\begin{proof}
    We have $F_3 F_2 \cong F_3 \otimes_\Q F_2$. The discriminants of $F_3$ and $F_2$ divide $27m^2$ and $4d$, respectively, and are hence coprime. \Cref{prop:index1} now implies that  $O_F = O_{F_3} \otimes_\Z O_{F_2}$.

    It is well-known that $O_{F_2} = \Z[\sqrt{-d}]$ when $d \equiv 1, 2 \pmod{4}$ and $O_{F_2} = \Z[(1+\sqrt{-d})/2]$ when $d \equiv 3 \pmod{4}$. It now easy to see that $O_{F_2} = \Z[\beta]$, and hence $O_F = O_{F_3}[\beta]$.
    
    Finally, if $d \equiv 1, 2 \pmod{4}$, then $\Tr_{F/F_3}(\beta) = -2d$ and $\N_{F/F_3}(\beta) = (d+1)d$. If $d \equiv 3 \pmod{4}$, then $\Tr_{F/F_3}(\beta) = -d$ and $\N_{F/F_3}(\beta) = cd$ where $c = (d+1)/4 \in \Z$.
\end{proof}

We establish an analogous result for $m = 2$ and arbitrary positive squarefree $d$. 

\begin{proposition}\label{prop:index-m=2}
    Let $d$ be a positive squarefree integer, and write $d = 2^{e_2}3^{e_3}d_0$ with $e_2, e_3 \in \{0,1\}$ and $\gcd(d_0,6) = 1$. Let $F_2 = \mathbb{Q}(\sqrt{-d})$, $F_3 = \mathbb{Q}(\sqrt[3]{2})$, and $F = F_3 F_2 = \mathbb{Q}(\sqrt[3]{2}, \sqrt{-d})$. Then $O_F = O_{F_3}[\beta]$ with
    \[ \beta = \begin{cases}  \displaystyle\frac{d+\sqrt{-d}}{\sqrt[3]{2}(\sqrt[3]{2}+1)^{e_3}} & \mbox{if $d \equiv 1,2 \pmod{4}$}, \\[10pt] \displaystyle\frac{d+\sqrt{-d}}{2(\sqrt[3]{2}+1)^{e_3}} & \mbox{if $d \equiv 3 \pmod{4}$}. \end{cases} \]
    Moreover, $\Tr_{F/F_3}(\beta)$ and $\N_{F/F_3}(\beta)$ are divisible by $d_0$ in $O_{F_3}$.
\end{proposition}
\begin{proof}
For brevity, let $\alpha = \sqrt[3]{2}$. The ring $O_{F_3} = \mathbb{Z}[\alpha]$  is monogenic and a PID; its discriminant is $-2^2 3^3$. So the only primes that ramify in $F_3$ are $2 = \alpha^3$ and $3 = (\alpha+1)^3\epsilon$ where $\epsilon = \alpha-1 \in O_{F_3}^\times$. Thus, $d = \alpha^{3e_2}(\alpha+1)^{3e_3}\epsilon^{e_3}d_0$, where $d_0$ is squarefree in~$O_{F_3}$. 

For any element $\kappa = \kappa_1 + \kappa_2 \sqrt{-d} \in F$, with $\kappa_1, \kappa_2 \in F_3$, put $\overline{\kappa} = \kappa_1 - \kappa_2\sqrt{-d}$. Then Tr$(\kappa) = \kappa + \overline{\kappa} = 2\kappa_1 = \alpha^3\kappa_1$ and N$_{F/F_3}(\kappa) = \kappa\overline{\kappa} = \kappa_1^2 + \kappa_2^2d$, respectively. 

We first show that $\Tr_{F/F_3}(\beta)$ and $\N_{F/F_3}(\beta)$ are divisible by $d_0$ in $O_{F_3}$. For brevity, write 
\[ \beta = \frac{d+\sqrt{-d}}{\alpha^f(\alpha+1)^{e_3}}\,, \qquad f = \begin{cases} 1 & \mbox{if $d \equiv 1, 2 \pmod{4}$}, \\ 3 & \mbox{if $d \equiv 3 \pmod{4}$}. \end{cases} \]
Noting that  $\alpha^{2f}$ divides $(d+1)d$ in $O_{F_3}$, and hence also divides $(d+1)\alpha^{3e_2}$, we obtain
\begin{align*} 
\Tr_{F/F_3}(\beta) &= \displaystyle \frac{\alpha^{3-f}d}{(\alpha+1)^{e_3}} = \alpha^{3-f+3e_2}(\alpha+1)^{2e_3}\epsilon^{e_3} d_0 \in d_0 O_{F_3}, \\[3pt]
\N_{F/F_3}(\beta) &= \displaystyle \frac{d^2+d}{\alpha^{2f}(\alpha+1)^{2e_3}} = (d+1)\alpha^{3e_2-2f}(\alpha+1)^{e_3} \epsilon^{e_3}d_0 \in d_0 O_{F_3}.
\end{align*}
Since the minimal polynomial of $\beta$ over $F_3$ is $f_\beta(T) = T^2 - \Tr_{F/F_3}(\beta)T + \N_{F/F_3}(\beta) \in O_{F_3}[T]$, this shows in particular that $\beta$ is integral over $O_{F_3}$, which implies $O_{F_3}[\beta] \subseteq O_F$.
To establish the inclusion $O_{F_3}[\beta] \supseteq O_F$, let 
\[ \kappa = \frac{\lambda_1}{\delta_1} + \frac{\lambda_2}{\delta_2}\beta \in O_F, \]
with $\lambda_1, \lambda_2, \delta_1, \delta_2 \in O_{F_3}$, $\delta_1 > 0, \delta_2 > 0$,  and $\gcd(\lambda_1, \delta_1) = \gcd(\lambda_2, \delta_2) = 1$. To show that $\kappa \in O_{F_3}[\beta]$, our task is to prove that $\delta_1 = \delta_2 = 1$. Write
\begin{equation} \label{eq:delta2kappa}
    \delta_2 \left (\kappa - \frac{\lambda_1}{\delta_1} \right )=  \lambda_2\beta. 
\end{equation}
Then $\delta_2(\kappa-\overline{\kappa}) = \lambda_2(\beta-\overline{\beta})$, so
\begin{equation*} 
\delta_2^2 \, \N_{F/F_3}(\kappa - \overline{\kappa}) = \lambda_2^2 \N_{F/F_3}(\beta-\overline{\beta}) =  \lambda_2^2\,\frac{4d}{\alpha^{2f}(\alpha+1)^{2e_3}} = \lambda_2^2\alpha^{6+3e_2-2f}(\alpha+1)^{e_3}\epsilon^{e_3}d_0.
\end{equation*}
The right hand side of this identity is squarefree except for the primes in $O_{F_3}$ dividing~$\lambda_2$ and possibly the power of $\alpha$. Since $\gcd(\lambda_2, \delta_2) = 1$, it follows that $\delta_2$ divides $\alpha^{3+e_2-f}$. 

If $d \equiv 3 \pmod{4}$, then $f = 3$ and $e_2 = 0$, so $\delta_2 \in O_{F_3}^\times$. Then \eqref{eq:delta2kappa} yields $\lambda_1/\delta_1 = \kappa - \lambda_2 \beta \in F_3 \cap O_F = O_{F_3}$, in which case $\gcd(\lambda_1, \delta_1) = 1$ forces $\delta_1 \in O_{F_3}^\times$. So $\kappa \in O_{F_3}[\beta]$.

Suppose now that $d \equiv 1, 2 \pmod{4}$, so $f = 1$. Let $v_\alpha()$ denote the $\alpha$-adic valuation on~$F_3$. Since $d^2+d \equiv 2 \pmod{4}$, we find that $v_\alpha(N_{F/F_3}(\beta)) = 1$. Multiplying \eqref{eq:delta2kappa} by $\delta_1$ and taking relative norms on $F/F_3$ yields 
\begin{equation} \label{eq:relnorms}
\delta_2^2 \, \N_{F/F_3}(\delta_1\kappa - \lambda_1) = \delta_1^2 \lambda_2^2 \,\text{N}_{F/F_3}(\beta). 
\end{equation}
If $\alpha$ divides $\delta_1$, then $\alpha$ does not divide $\lambda_1$, so $v_\alpha(\N_{F/F_3}(\delta_1\kappa - \lambda_1)) = 0$. Then the left hand side of \eqref{eq:relnorms} has even $\alpha$-adic valuation, while the $\alpha$-adic valuation of the right hand side is odd, a contradiction. It follows that $v_\alpha(\delta_1) = 0$.

If $\alpha$ divides $\delta_2$, then $\alpha$ does not divide $\lambda_2$. But then the left hand side of \eqref{eq:relnorms} has $\alpha$-adic valuation at least 2, whereas the right hand side has $\alpha$-adic valuation 1, which is again a contradiction. So $v_\alpha(\delta_2) = 0$. Since $\delta_2$ divides $\alpha^{3+e_2-f}= \alpha^{2+e_2}$, this forces $\delta_2 \in O_{F_3}^\times$. 
Then $\lambda_1\delta_1^{-1} = \kappa- \lambda_2  \delta_2^{-1}\beta \in O_F \cap F_3 = O_{F_3}$. Since $\gcd(\lambda_1, \delta_1) = 1$, we must have $\delta_1 \in O_{F_3}^\times$, so~$\kappa \in O_F$. 
\end{proof}

\subsection{Orthogonality of $\Lambda_F$ for $F = \mathbb{Q}(\sqrt[3]{D^3 \pm 1}, \sqrt{-d})$} \label{sec:orth}

We now restrict to integers $m$ such that $m \pm 1$ is a perfect cube for one of the choices of sign. Stender \cite{stender75} explicitly gave the fundamental units of the corresponding family of pure cubic fields as follows.

\begin{proposition} \label{prop:stender_cubic}
    Let $D$ be a positive integer and $m = D^3 \pm 1 > 0$ cubefree. Then $\epsilon_t = \sqrt[3]{m} - D$ is a fundamental unit of $F_3 = \mathbb{Q}(\sqrt[3]{m})$ unless $(D,m) = (3,28)$, in which case $\sqrt{\epsilon_t} = (\sqrt[3]{28}^2 - 2\sqrt[3]{28}-2)/6$ is a fundamental unit of $F_3 = \mathbb{Q}(\sqrt[3]{28})$. 
\end{proposition}

Let $F_3$ be a field belonging to the Stender family of \Cref{prop:stender_cubic}, and $F = F_3(\sqrt{-d})$ with $d > 0$ squarefree. Assuming certain conditions on $m$ and $d$, we construct a quartic polynomial for which the non-existence of a root modulo an appropriate factor of~$d$ guarantees the orthogonality of the log unit lattice $\Lambda_F$. 

\begin{proposition} \label{prop:orth}
Let $D$ and $m$ be as in \Cref{prop:stender_cubic}, and let $d$ be a positive squarefree integer. Assume that one of the following holds.
\begin{itemize}
\item $m = 2$ or
\item $m$ is odd, $\gcd(m,d) = 1$ and $3 \nmid d$.
\end{itemize}
Let $d_0 = d/\gcd(d,6)$ if $m = 2$ and $d_0 = d$ otherwise. If the polynomial 
\[ f(T) = T^4 + 8D^2T^3 + 18mDT^2 + 8m^2T + m^2 D^2 \]
has no integer root modulo $d_0$, then the log unit lattice $\Lambda_F$ of $F = \mathbb{Q}(\sqrt[3]{m}, \sqrt{-d})$ is orthogonal.
\end{proposition}
\begin{proof}
Let $F_3 = \Q(\sqrt[3]{m})$. The lattices $\Lambda_{F/F_3}^t$ and $\Lambda_{F/F_3}^r$ as given in \Cref{lem:vtvrortho} both have rank~1. Let $v_t$ and $v_r$ be generators of these respective lattices.

Suppose by contradiction that $\Lambda_F$ is not orthogonal. Then by \Cref{cor: characterization of orthogonality}, there exists  $\epsilon_r\in \Log_F^{-1}(v_r)$ such that $\sqrt{\epsilon_t \epsilon_r} \in F$, where $\epsilon_t$ is as in \Cref{prop:stender_cubic} and $\epsilon_r$ is a unit of $O_F$ such that $\Log_F(\epsilon_r)$ is a generator of $\Lambda_{F/F_3}^r$. By \Cref{remark: definition of vr and vt} and \Cref{prop:stender_cubic}, 
\begin{equation} \label{eq:norm-sqrt1}
\N_{F/F_3}(\sqrt{\epsilon_t \epsilon_r}) = \pm \epsilon_t = \pm(\sqrt[3]{m}-D).
\end{equation}
We have $O_F = O_{F_3}[\beta]$ where $\beta$ is given in \Cref{lem:index1} when $m$ is odd and in \Cref{prop:index-m=2} when $m = 2$. Write $\sqrt{\epsilon_t \epsilon_r} = \theta_1 + \theta_2\beta$ with $\theta_1, \theta_2 \in O_{F_3}$. By \Cref{lem:index1} and \Cref{prop:index-m=2}, 
\begin{equation} \label{eq:norm-sqrt2}
    \N_{F/F_3}(\sqrt{\epsilon_t \epsilon_r}) \equiv \theta_1^2 \pmod{d_0}.
\end{equation}
Let $\alpha = \sqrt[3]{m}$ for brevity, and let $I = [O_{F_3} : \Z[\alpha]]$. Then $I = 1$ if $m=2$ and $I$ divides $3m$ if~$m$ is odd. Either way, we have $\gcd(I,d) = 1$. Write $\theta_1 = \frac{1}{I}(x + y\alpha + z\alpha^2)$ with $x, y, z \in \Z$. Then \eqref{eq:norm-sqrt1} and \eqref{eq:norm-sqrt2} yield
\begin{align}
    x^2 + 2myz & \equiv \mp I^2 D \pmod{d_0} , \label{eq:1} \\
    2xy + mz^2 & \equiv \pm I^2 \pmod{d_0}, \label{eq:2} \\
    2xz + y^2  & \equiv 0 \pmod{d_0} . \label{eq:3}
\end{align}
We multiply \eqref{eq:2} by $D$ and add the result to \eqref{eq:1} to obtain
\begin{equation} \label{eq:4}
    x^2 + Dmz^2 \equiv -2y(mz+Dx) \pmod{d_0}.
\end{equation}
Squaring \eqref{eq:4} and substituting \eqref{eq:3} into the result yields
\begin{equation} \label{eq:5}
     (x^2 + Dmz^2)^2 \equiv -8xz(mz+Dx)^2 \pmod{d_0}.
\end{equation}
We claim that $\gcd(z,d_0) = 1$. To that end, suppose that some prime $p$ divides $\gcd(z,d_0)$. Then $p$ divides $y$ by \eqref{eq:3}, and hence $p$ also divides $I$ by \eqref{eq:2}, contradicting $\gcd(I, d_0) = 1$. 

Put $w \equiv xz^{-1} \pmod{d_0}$. Then multiplying \eqref{eq:5} by $z^{-4} \pmod{d_0}$ and rearranging terms yields $f(w) \equiv 0 \pmod{d_0}$, where $f(T)$ is the polynomial in the statement of \Cref{prop:orth}. Hence $w$ is an integer root of $f(T)$ modulo $d_0$ which contradicts the assumption.   
\end{proof}

\subsubsection{Fixed $d$ and varying $D$}

Via Proposition \ref{prop:orth}, we can establish the existence of infinite families of pure sextic fields with orthogonal log unit lattices. We construct for every fixed squarefree $d$ coprime to 6 an infinite family of sextic fields $F \supset \mathbb{Q}(\sqrt{-d})$ whose log unit lattices are orthogonal.

\begin{proposition} \label{prop:fixd_varym} 
Let $d$ be a fixed positive squarefree integer coprime to $6$. Let $D= kd+1$, where $k\in \mathbb{Z}_{>0}$ is odd. Let $m= D^3+1$, and suppose that $m$ is cubefree. If the polynomial
\begin{equation}\label{equation: The polynomial}
f(T)=T^4+8T^3+ 36T^2+ 32T+4
\end{equation}
has no integer root modulo $d$, then the log unit lattice $\Lambda_F$ of $F= \mathbb{Q}(\sqrt[3]{m}, \sqrt{-d})$ is orthogonal.
\end{proposition} 
\begin{proof}
We have $D \equiv 1 \pmod{d}$ and hence $m \equiv 2 \pmod{d}$. Thus, modulo $d$, the polynomial $f(T)$ of \Cref{prop:orth} becomes the polynomial $f(T)$ in \eqref{equation: The polynomial}. The result now follows from \Cref{prop:orth}, subject to verifying the required assumptions on $d$ and $m$ stated there.

By assumption, $d$ is positive, squarefree and coprime to 6, so $3 \nmid d$. Since $k$ and $d$ are both odd, $D$ is even and hence $m$ is odd. Moreover, $m$ is cubefree by assumption. Finally, since $m \equiv 2 \pmod{d}$, any common factor of $d$ and $m$ divides both 2 and $d$ (which is odd), implying that $\gcd(d,m) = 1$.
\end{proof}

\begin{remark}
    Note that if $f(T)$ as given in \Cref{prop:fixd_varym} has an integer root modulo $d$, then we cannot draw any conclusion about orthogonality. In fact, \Cref{prop:fixd_varym} is not an ``if and only if" result. For example, for $d=19$ and $m = (1\cdot 19 +1)^3 + 1 = 8001$, the polynomial $f(T)$ has a root modulo $d$, but the corresponding sextic log unit lattice is nevertheless orthogonal.
\end{remark}

\begin{corollary} \label{cor:fixd_varym}
    There are infinitely many positive squarefree integers $d$ 
    such that for each such $d$ there exists an infinite family of fields $F= \mathbb{Q}(\sqrt[3]{m},  \sqrt{-d})$ with orthogonal log unit lattices. Specifically, for every positive squarefree integer $d$ coprime to $6$ that is a multiple of any of the primes belonging to the set $S = \{ 5, 11, 17, 23$, $31, 43, 47, 59, 89\}$, the infinite family of fields $F = \mathbb{Q}(\sqrt[3]{(kd+1)^3+1}, \sqrt{-d})$ satisfying the conditions of Proposition \ref{prop:fixd_varym} has orthogonal log unit lattices.   
\end{corollary}
\begin{proof}
  Using \cite[Theorem 1.1]{booker2016squarefree} (see also \cite{ricci33}) shows that there are infinitely many $k$ such that $(kd+1)^3+1$ is cubefree. 
  It is easily verified that the polynomial $f(T)$ of \Cref{prop:fixd_varym} does not have an integer root modulo any of the primes in $S$, and hence does not have an integer root modulo any multiple of an element of~$S$. 
\end{proof}

\subsubsection{Varying $d$ and $D=1$}

For $D = 1$ and $m = D^3 + 1 = 2$, we again construct an infinite family of purely sextic fields $F \supset \mathbb{Q}(\sqrt[3]{2})$ with orthogonal log unit lattices. The results in this case are analogous to \Cref{prop:orth} and \Cref{cor:fixd_varym}, except that there are no divisibility restrictions on $d$. 

\begin{proposition} \label{prop:m=2varyd}
Let $d$ be any positive squarefree integer. If the polynomial
\[f(T)=T^4+8T^3+ 36T^2+ 32T+4\]   
has no integer root modulo $d$, then the log unit lattice $\Lambda_F$ of $F= \mathbb{Q}(\sqrt[3]{2}, \sqrt{-d})$ is orthogonal.
\end{proposition}
\begin{proof}
     Let $d_0 = d/\gcd(d,6)$. Since 0 and 1 are integer roots of $f(T)$ modulo 2 and 3, respectively, $f(T)$ has no integer root modulo $d_0$ if and only if it has no integer root modulo~$d$. The result now follows from \Cref{prop:orth}. 
\end{proof}  

\begin{remark}
    Note that if the polynomial $f(T)$ defined in \Cref{prop:m=2varyd} possesses an integer root modulo $d$, no conclusion can be drawn regarding orthogonality. Consequently, \cref{prop:m=2varyd} does not provide a necessary and sufficient condition. For example, setting $m = 1^2 + 1 = 2$ with $d = 71$ or $d = 83$ yields a polynomial $f(T)$ with roots modulo $d$, even though the corresponding sextic log unit lattices are orthogonal.
\end{remark}

\begin{corollary} \label{cor:m=2varyd}
There exists an infinite family of fields $F= \mathbb{Q}(\sqrt[3]{2}, \sqrt{-d})$ with orthogonal log unit lattices. Specifically, if $d$ varies over the squarefree multiples of any of the primes belonging to the set $S = \{ 5, 11, 17, 23$, $31, 43, 47, 59, 89\}$, then the infinite family of fields $F = \mathbb{Q}(\sqrt[3]{2}, \sqrt{-d} )$ has orthogonal log unit lattices. 
\end{corollary} 

The polynomial $f(T)$ in Propositions \ref{prop:fixd_varym} and \ref{prop:m=2varyd} is irreducible over $\mathbb{Z}$. We note that the list $S$ of primes in Corollaries \ref{cor:fixd_varym} and \ref{cor:m=2varyd} was obtained via numerical verification and is by no means exhaustive.

\subsection{Density of orthogonal $\Lambda_F$ for $F = \mathbb{Q}(\sqrt[3]{2}, \sqrt{-d})$} \label{sec:density}

Let $\mathcal{F}_{\sqrt[3]{2}}$ denote the family of fields $F=\mathbb{Q}(\sqrt[3]{2}, \sqrt{-d})$ of \Cref{prop:m=2varyd}, where $d$ is a square-free positive integer. For any prime $p \ge 5$, we compute the proportion of fields $F = \mathbb{Q}(\sqrt[3]{2}, \sqrt{-p}) \in \mathcal{F}_{\sqrt[3]{2}}$ whose log unit lattice is orthogonal. This is accomplished by characterizing the primes $p$ for which the polynomial $f(T)$ of \Cref{prop:m=2varyd} has no integer root modulo $p$ via the Artin symbol in the Galois group of $f(T)$. We also derive a lower bound on the proportion of all fields in $\mathcal{F}_{\sqrt[3]{2}}$ whose log unit lattice is orthogonal.

\begin{proposition} \label{prop: cycletype}
Let $\alpha \in \mathbb{C}$ be a root of $f(T)$ as in \Cref{prop:m=2varyd}. Let $K_4 = \mathbb{Q}(\alpha)$ and $\tilde{K}_4$ the Galois closure of $K_4$. Let $p \ge 5$ be a prime. Then $f(T)$ has no integer root modulo $p$ if and only if there exists a prime ideal $\mathfrak{P}$ in $O_{\tilde{K}_4}$ above $p$ such that the Artin symbol $\left[ \frac{\tilde{K}_4/\mathbb{Q}}{\mathfrak{P}} \right]$, viewed as an element in $S_4$, is either a product of two disjoint $2$-cycles or a $4$-cycle. 
\end{proposition}
\begin{proof}
The discriminant of $f(T)$ is $-2^{16}\cdot 3^5$, so $p$ is unramified in $K_4$ and does not divide the index $[O_{K_4}:\mathbb{Z}[\alpha]]$. By Kummer's factorization theorem (see \cite[Theorem 7.6, Chapter I]{janusz1973algebraic} for example), $f(T)$ has no integer root modulo $p$ if and only if all prime ideals above $p$ in $K_4$ have residue degree exceeding 1. Equivalently, $p$ is inert or splits into two distinct prime ideals of residue degree $2$ in $K_4$. 

We have $\Gal(\tilde{K}_4/\mathbb{Q}) \cong S_4$ and $\Gal(\tilde{K}_4/K_4) \cong S_3$. Every element of $\Gal(\tilde{K}_4/\mathbb{Q})$ acts on the cosets of $\Gal(\tilde{K}_4/\mathbb{Q})/\Gal(\tilde{K}_4/K_4)$ as a permutation and as such has a decomposition into disjoint cycles. By \cite[Chapter III, Prop.~2.8]{janusz1973algebraic},  the lengths of the cycles in the cycle decomposition of the Artin symbol $\left[ \frac{\tilde{K}_4/\mathbb{Q}}{\mathfrak{P}} \right]$ acting on these cosets correspond to the residue degrees of the prime ideals above $p$ in $K_4$. It follows that $p$ has the desired decomposition in $K_4$ if and only if $\left[ \frac{\tilde{K}_4/\mathbb{Q}}{\mathfrak{P}} \right]$ is a 4-cycle or a product of two disjoint 2-cycles. 
\end{proof}

\begin{corollary}\label{cor:density_prime}
    The density of primes $p$ for which the log unit lattice of $F=\mathbb{Q}(\sqrt[3]{2}, \sqrt{-p})$ is orthogonal is at least $3/8$.
\end{corollary}
\begin{proof}
   Let $S$ be the set of permutations in $S_4$ that are either 4-cycles or products of two disjoint 2-cycles. Then $S$ has cardinality $\#S = 9$, as the number of 4-cycles in $S_4$ is $3! = 6$ and the number of products of two disjoint 2-cycles in $S_4$ is $\binom{4}{2}/2 = 3$. By \Cref{prop:m=2varyd} and \Cref{prop: cycletype}, the log unit lattice of a field $F=\mathbb{Q}(\sqrt[3]{2}, \sqrt{-p})$, with $p$ prime, is orthogonal if there exists an $O_{\tilde{K}_4}$-prime ideal $\mathfrak{P}$ above $p$ such that the Artin symbol $\left[ \frac{\tilde{K}_4/\mathbb{Q}}{\mathfrak{P}} \right]$ in $S_4$ belongs to~$S$. By the Chebotarev Density Theorem (see \cite[Chapter~V, Theorem~10.4]{janusz1973algebraic}, the density of such primes is $\#S/\#S_4 = 9/24=3/8$. 
\end{proof}

The set $\mathcal{F}_{\sqrt[3]{2}}$ is in one-to-one correspondence with square-free positive integers $d$, which allows us to endow $\mathcal{F}_{\sqrt[3]{2}}$ with an ordering. We compute a lower bound on the density of fields $F \in \mathcal{F}_{\sqrt{2}}$ whose log unit lattices are orthogonal.

\begin{theorem} \label{thm:2dens}
    At least $68\%$ of  fields in $\mathcal{F}_{\sqrt[3]{2}}$ have orthogonal log unit lattices. 
\end{theorem}
\begin{proof}
      If $d$ is squarefree, then by the Chinese Remainder Theorem, a polynomial in $\mathbb{Z}[T]$ has no root modulo $d$ if and only if it has no root modulo some prime divisor of $d$. Let~$A$ be the set of primes $p$ such that the polynomial $f(T)=T^4+8T^3+ 36T^2+ 32T+4$ of \Cref{prop:m=2varyd} has no root modulo $p$. Then by Proposition \ref{prop:m=2varyd}, for all squarefree multiples $d$ of some prime in $A$, the unit log lattice~$\Lambda_F$ of $F = \mathbb{Q}(\sqrt[3]{2}, \sqrt{-d})$ is orthogonal. Via identifying log unit lattices $\Lambda_F$  with squarefree natural numbers $d$ as above, a lower bound on the density of orthogonal unit log lattices $\Lambda_F$ with $F \in \mathcal{F}_{\sqrt[3]{2}}$ is thus given by the density of squarefree positive integers divisible by at least one prime in $A$. We compute this density by resorting to \cite[Theorem~1]{brown2021naturaldensitysetssquarefree}, which states that for disjoint sets~$P$ and~$T$ of prime numbers with~$T$ finite, the natural density of squarefree positive integers which are divisible by all the primes in $T$ and by none of the primes in $P$ is 
     $$\frac{6}{\pi^2}\prod_{p\in T}\frac{1}{p+1}\prod_{p\in P}\frac{p}{p+1}.$$
     
     In our application of this result, $T$ will be the empty set. Let $D$ be the set of squarefree natural numbers and $D_P \subset D$ the set of squarefree integers divisible by no prime in $A$. For any $n\in\mathbb{N}$ and any subset $X\subseteq \mathbb{N}$, define $X_n=X\cap\{1,\dots,n\}$. Then
     \[ \lim_{n\rightarrow\infty}\frac{\#D_n}{n}=\frac{6}{\pi^2}, \qquad \lim_{n\rightarrow\infty}\frac{\#D_{P,n}}{n}=\frac{6}{\pi^2}\prod_{p\in A}\frac{p}{p+1}. \]
     Let $D_A$ be set of squarefree natural numbers divisible by at least one prime in $A$. We wish to find the density of $D_A$ in $D$, which we can write as
     \[
         \lim_{n\rightarrow\infty}\frac{\#D_{A,n}}{\#D_n}=\lim_{n\rightarrow\infty}\frac{n}{\#D_n}\cdot\lim_{n\rightarrow\infty}\frac{\#D_{A,n}}{n}\\
         =\frac{\pi^2}{6}\lim_{n\rightarrow\infty}\frac{\#D_{A,n}}{n}.
     \]
     Since $D = D_A \sqcup D_p$, we have 
    \[
         \lim_{n\rightarrow\infty}\frac{\#D_{A,n}}{n} = \lim_{n\rightarrow\infty}\frac{\#D_n}{n} -\lim_{n\rightarrow\infty}\frac{\#D_{P,n}}{n} = \frac{6}{\pi^2} \left ( 1 - \prod_{p\in A}\frac{p}{p+1} \right).
      \]   
    Altogether we find that 
     $$\lim_{n\rightarrow\infty}\frac{\#D_{A,n}}{\#D_n}=1-\prod_{p\in A}\frac{p}{p+1}.$$
      The lower bound on this density is obtained by replacing $A$ in the above product by the set $A_0$ of all primes $p$ with  $p \le 10^{12}$ such that $f(T)$ as given in Proposition \ref{prop:m=2varyd} has no roots modulo $p$. Numerical computation shows that $1 - \prod_{p\in A_0}\frac{p}{p+1} \approx 0.6806$.
\end{proof}

\section{Conclusion and future research}\label{sec:conclusion}
There is a considerable amount of prior literature on the geometry and shapes of the lattices arising from number rings. In comparison, research into log unit lattices lags behind, offering numerous further promising avenues for inquiry. Our work herein leaves open a number questions for future investigation. A proof of \Cref{conj:remaining Galois groups} would complete the unconditional classification of rank $2$ unit shapes.
 Removing the dependence on the Weak Schanuel Conjecture (\Cref{conj: algebraic independence of logs}) and the Four Exponentials Problem Conjecture (\Cref{conj: four exponentials problem}) in the respective proofs of  Theorems \ref{theorem: classification or rank 2} and \ref{theorem: Completeness of the shape} would establish these results unconditionally. Another important next step would be to ascertain whether or not an infinite family  of well-rounded log unit lattices exists and, if so, whether and how they can be explicitly constructed. The possibility of extending the result of \Cref{theorem: Completeness of the shape} to other Galois groups listed in \Cref{prop:Galois groups rank 2}, in order to determine whether or not the unit shape determines the fields in the corresponding families uniquely, clearly warrants further investigation. It would also be interesting to refine the result of part 2 of \Cref{theorem: classification or rank 2} by finding an upper bound on the number of fields with a fixed index 2 subfield whose unit shapes lie on the lower arc of $\mathcal{S}_2$. Finally, generalizing our methodology for investigating the geometry and shapes of log unit lattices of rank $2$ to higher unit rank represents a compelling direction for future research. Numerous questions on the density and equidistribution of unit shapes within $\mathcal{S}_m$ remain open even for the case $m=2$.

\bibliographystyle{alpha}
\bibliography{refs}
 
\end{document}